\documentclass[11pt]{amsart}
\usepackage{a4wide,amsfonts,graphicx,
amssymb,stmaryrd,amscd,amsmath,latexsym,amsbsy,amsthm}
\usepackage[all,knot,poly]{xy}
\usepackage[all,knot,poly]{xy}
\usepackage{color}
\usepackage{ifpdf}
\ifpdf 
  \usepackage{graphicx}
  \DeclareGraphicsExtensions{.pdf,.png,.mps}
  \usepackage{pgf}
  \usepackage{tikz}
\else 
  \usepackage{graphicx}
  \DeclareGraphicsExtensions{.eps,.bmp}
  \usepackage{pgf}
  \usepackage{tikz}
  \usepackage{pstricks}
\fi
\usepackage{epic,bez123}
\usepackage{floatflt}
\usepackage{wrapfig}

\newcommand{\Z}{\mathbb{Z}}

\newcommand{\C}{\mathbb{C}}

\newcommand{\isom}{\xrightarrow{\;\sim\;}}

\def\beq{\begin{equation}}
\def\eeq{\end{equation}}

\def\a{\alpha}
\def\b{\beta}

\def\d{\delta}
\def\e{\varepsilon}

\def\A {\mathcal A}

\def\H{\mathcal H}
\def\G{\mathcal G}
\def\O {\mathcal O}
\def\M{\mathcal M}

\def\F{\mathbf F}

\newenvironment{res}
               {\begin{equation}
\begin{minipage}{0.85\textwidth}}
               { \end{minipage}\end{equation} }
\def\ber{\begin{res} }
\def\eer{\end{res}}

\numberwithin{equation}{section}
\newtheorem{thm}{Theorem}[section]

\newtheorem{lemma}[thm]{Lemma}
\newtheorem{lem}[thm]{Lemma}

\newtheorem{cor}[thm]{Corollary}

\newtheorem{prop}[thm]{Proposition}

\newtheorem{dfn}[thm]{Definition}
\newtheorem{rem}[thm]{Remark}

\makeatletter
\def\section{\@startsection {section}{1}{\z@}{3.5ex plus 1ex minus
    .2ex}{2.3ex plus .2ex}{\large\bf}}
    \def\subsection{\@startsection{subsection}{2}{\z@}{3.25ex plus 1ex minus
 .2ex}{1.5ex plus .2ex}{\bf}}
\makeatother

\def\a{\alpha}
\def\b{\beta}
\def\l{\lambda}

\def\e{\epsilon}
\def\ga{\gamma}

\def\d{\delta}

\def\cc{C^\infty}

\def\af{\mathfrak{a}}

\def\het{\operatorname{ht}}

\def\q{{\bf q}}
\def\s{\mathfrak{s}}
\def\r{\mathfrak{r}}
\def\si{\sigma}

\def\Wf{\mathfrak{W}}

\def\A{\mathcal{A}}

\def\Nc{\mathcal{N}}

\def\O{\mathcal{O}}

\def\Ze{\mathcal{Z}}

\def\M{\mathcal{M}}

\def\A{\mathcal{A}}

\def\H{\mathcal{H}}
\def\F{\mathcal{F}}
\def\J{\mathcal{J}}
\def\K{\mathcal{K}}

\def\P{\mathcal{P}}
\def\S{\mathcal{S}}
\def\V{\mathcal{V}}
\def\W{\mathcal{W}}
\def\Ri{\mathcal{R}}
\def\Rf{\mathfrak{R}}

\def\rnr{R_{\mathrm{nr}}}

\def\i{\iota}

\makeindex
\begin{document}
\title[Analytic R-groups]{Analytic R-groups of affine Hecke algebras}
\author{Patrick Delorme}
\address{Institut de Math\'ematiques de Luminy, UPR 9016 du
CNRS\\ Facult\'e des Sciences de Luminy\\ 163 Avenue de Luminy,
Case 901, 13288 Marseille Cedex 09\\ France\\ email:
delorme@iml.univ-mrs.fr}
\thanks{During the preparation of this paper
the first named  author was supported by the program ANR-08-BLAN-01}
\author{Eric Opdam}
\address{Korteweg de Vries Institute for Mathematics\\
University of Amsterdam\\P.O. Box 94248\\1090 GE Amsterdam\\
The Netherlands\\ email: e.m.opdam@uva.nl}
\date{\today}
\thanks{It is a pleasure to thank Dan Ciubotaru, David Kazhdan
and Birgit Speh for useful comments}
\maketitle
\tableofcontents
\section{Introduction}
Let $\H$ be an affine Hecke algebra in the sense of Lusztig (see
\cite{Lu} or \cite{O}). The Fr\'echet algebra $\S$, the Schwartz algebra
completion of $\H$, is a central object of study for the
harmonic analysis of $\H$.
The key to understanding the structure of $\S$
is the Fourier transform isomorphism $\F$ which
(by the main result of \cite[Theorem 5.3]{DO}) identifies $\S$ with the algebra
of Weyl group invariant sections of a certain smooth endomorphism bundle
over the space $\Xi_u$ of ``tempered standard induction data.''

Such a tempered standard induction datum consists of a triple $(P,\delta,t)$ where $P$
denotes a subset of the set of simple roots of the based root system underlying $\H$.
This subset $P$ defines a ``standard Levi subalgebra'' $\H^P$ of $\H$, with semisimple
quotient $\H_P$, and $\delta$ denotes a discrete series representation of $\H_P$.
Finally $t$ is a unitary induction parameter for $\H_P$, which is used to lift
$\delta$ to a unitary representation $\delta_t$ of $\H^P$. The set of such unitary
parameters has the structure of a compact real torus.
Thus $\Xi_u$ is a finite union of compact tori.

The endomorphism bundle alluded to above is constructed from a canonical projective
unitary representation $\pi$ of the groupoid $\W_{\Xi_u}$
of tempered standard induction data (cf. \cite[Section 3.5]{DO}).
The arrows in this groupoid are twists by certain isomorphisms of the
$\H_P$ defined in terms of the (affine) Weyl group.
This groupoid $\W_{\Xi_u}$ has been determined explicitly in general
if $\H$ is of simple type, see \cite{OpSo}.
The projective representation $\pi$ yields a $2$-cocycle
$\gamma\in Z^2(\W_{\Xi_u},\operatorname{U}(1))$ of $\W_{\Xi_u}$ with values
in $\operatorname{U}(1)$.

The results \cite[Theorem 3.11, Theorem 3.19]{DO} state that
for every $\xi\in\Xi_u$ the representation $\pi(\xi)$ of $\H$
is unitary and tempered,
and that for every irreducible tempered module $\rho$ of $\H$ there exists a
unique orbit $\W\xi$ (with $\xi\in\Xi_u$) such that $\rho$ is equivalent to
an irreducible summand of $\pi(\xi)$.

The main purpose of this article is to decompose $\pi(\xi)$ for
$\xi\in\Xi_u$. For this reason we introduce a certain subgroup
$\Rf_\xi$, called the $R$-group, of the isotropy subgroup
$\W_{\xi,\xi}$ of $\xi$ in the Weyl groupoid $\W$. We show an analog of the
Knapp-Stein Dimension Theorem \cite{KS}, \cite{S} which says that the
commutant of the representation $\pi(\xi)$ has a basis given by the
$\pi(\mathfrak{r})$ with $\mathfrak{r}$ running over the elements of
$\Rf_\xi$. As in \cite[Section 2]{A} it follows that there exists a
bijective correspondence between the irreducible representations
of $\H$ arising as summands of $\pi(\xi)$ on the one hand, and
the irreducible representations of the twisted complex group algebra
${}^{\gamma_\xi}\mathbb{C}[\Rf_\xi]$ on the other hand (where
$\gamma_\xi$ denotes the restriction of $\gamma$ to
$\Rf_\xi$). This is the content of our main result,
Theorem \ref{thm:Rmorequiv}.
If $\H$ is of simple type, our results yield the
classification of the (equivalence classes of) irreducible
tempered modules of $\H$. In the ``equal parameter case'' \cite{KL}
and more generally for the affine Hecke algebras arising in the context of
``unipotent representations'' of inner forms of simple adjoint split groups,
a classification of the tempered irreducible modules
in terms of geometric data is also known \cite{Lu5}.

The analogy with the theory of tempered representations of the group $G$
of points of a reductive group defined over a local field
is not surprising, since for various specializations of the parameters
of $\H$ it is known that its module category is equivalent
to a Bernstein block $\mathfrak{b}$ in the category of smooth representations of
such a group $G$ (\cite{Lu4}, \cite{Mo1}, \cite{Mo2}, \cite{MP}, \cite{Hei}).
If $\H$ is in fact isomorphic to the Hecke algebra of a ``$\mathfrak{b}$-type''
then this equivalence is known to respect temperedness and Plancherel measures
\cite{BHK}. Hence in this context our results
yield the classification of the irreducible tempered
representations of $G$ which belong to $\mathfrak{b}$.

It is a fundamental question how the $R$-groups
of parabolic induction for $\H$ which we will define below are related to
the $R$-groups of parabolic induction for $G$ if $\H$ is the Hecke algebra of a type for
$\mathfrak{b}$ or in the context of \cite{Hei}. Some general results in this
direction have been achieved by Roche \cite{Ro}

The $R$-groups and 2-cohomology classes $[\ga_\xi]$ for classical affine
Hecke algebras are amenable to \emph{direct} explicit computation.
This is illustrated
by Slooten's computation \cite{Sl1} of the $R$-groups for classical
Hecke algebras when the inducing representation is discrete series with
real infinitesimal character (in the sense of \cite{BM}),
and by the results in Section \ref{sec:coc}
of the present paper, proving the triviality of the 2-cocyles $\ga_\xi$
for classical Hecke algebras in all cases. With these results at hand,
our decomposition theorem amounts in these cases to the proof of Slooten's conjectural
classification \cite[Conjecture 4.3(i)]{Sl1} of the irreducible
tempered representations with real central character for classical
Hecke algebras (see \cite{Sl1}).
For a \emph{geometric} approach to these results, see \cite{ChKa}.


The main technical thrust of the proof of Theorem
\ref{thm:Rmorequiv} is the fact that the
$\W$-average of the product of a smooth section of the endomorphism bundle
with the $c$-function is itself a smooth section (see equation (\ref{eq:av})).
Another technical tool is the computation of the constant term for
generic parameters \cite[Section 6.2]{DO}.
\section{Affine Hecke algebras}
The structure of an affine Hecke algebra $\H=\H(\Ri,q)$ is
determined by an affine root datum (with basis) $\Ri$ together
with a label function $q$ defined on the extended affine Weyl
group $W$ associated to $\Ri$. We refer the reader to
\cite{Lu},\cite{O},\cite{DO} for the details of the definition
of the algebra $\H(\Ri,q)$, which we will only briefly review
here.

Let $\Ri=(X,Y,R_0,R_0^\vee,F_0)$ be a root datum (with basis
$F_0\subset X$ of simple roots of $R_0\subset X$). Let $W_0$ denote the
Weyl group of the reduced integral root system $R_0$. The extended
affine Weyl group $W$ associated with $\Ri$ is by definition
$W=W_0\ltimes X$. The affine root system $R$ is equal to
$R:=R_0^\vee\times\mathbb{Z}\subset Y\times\mathbb{Z}$. Observe that
$R$ is closed for the natural action of $W$ on the set of integral
affine linear functions $Y\times\mathbb{Z}$ on $X$.
Furthermore $R$ is the disjoint union of the positive and the negative
affine roots $R=R_+\cup R_-$ as usual, and we define the length
function $l$ on $W$ by
\begin{equation}
l(w):=|R_+\cap w^{-1}R_-|.
\end{equation}
The affine simple roots are denoted by $F^{\mathrm{aff}}$.

A label function $q:W\to\mathbb{R}_+$ is a function which is length
multiplicative (i.e. $q(uv)=q(u)q(v)$ if $l(uv)=l(u)+l(v)$) and
which in addition satisfies $q(\omega)=1$ if $l(\omega)=0$.
Thus a label function is completely determined by its values on
the set $S^{\mathrm{aff}}$ of affine simple reflections in $W$.
Observe that this gives rise to a positive function on $S^{\mathrm{aff}}$
which is constant on $W$-conjugacy classes of simple reflections,
and conversely, every such function gives rise to a label
function.

We choose a base
$\q>1$ and define $f_s\in\mathbb{R}$ such that $q(s)=\q^{f_s}$
for all $s\in S^{\mathrm{aff}}$.

Given these data, the affine Hecke algebra $\H=\H(\Ri,q)$ is
described as follows. It is the unique complex unital algebra with
basis $N_w$ ($w\in W$) over $\mathbb{C}$ subject to the following
relations (here $q(s)^{1/2}$ denotes the positive square root of $q(s)$):
\begin{enumerate}
\item $N_{uv}=N_uN_{v}$
for all $u,v\in W$ such that
$l(uv)=l(u)+l(v)$.
\item
$(N_s+q(s)^{-1/2})(N_s-q(s)^{1/2})=0$
for all $s\in S^{\mathrm{aff}}$.
\end{enumerate}
\subsubsection{Root labels for the non-reduced root system}
The label function $q$ on $W$ can also be defined in terms
of root labels. We associate a possibly non-reduced root
system $\rnr$ with $\Ri$ by
\begin{equation}
\rnr:=R_0\cup\{2\a\mid \a^\vee\in R_0^\vee\cap 2Y\}.
\end{equation}
Observe that $a+2\in Wa$
for all $a\in R$, but that $a+1\in Wa$ iff $a=\a^\vee + n$ with
$2\a\not\in\rnr$.

Let $R\ni a\to q_a$ be the unique $W$-invariant function on $R$ such that
$q_{a+1}:=q(s_{a})$ for all $a\in F^{\mathrm{aff}}$.
Now for $\a=2\b\in \rnr\backslash R_0$ we define
\begin{equation}
q_{\a^\vee}:=\frac{q_{\b^\vee+1}}{q_{\b^\vee}}.
\end{equation}
It is easy to see that in this way the set of positive $W_0$-invariant
functions $\a^\vee\to q_{\a^\vee}$ on $\rnr^\vee$ corresponds bijectively to
the set of label functions $q$ on $W$.
With these conventions we have for all $w\in W_0$
\begin{equation}
q(w)=\prod_{\alpha\in R_{\mathrm{nr},+}\cap
w^{-1}R_{\mathrm{nr,-}}} q_{\alpha^\vee}.
\end{equation}

We denote by $R_1\subset X$ the following
reduced root subsystem of $\rnr$:
\begin{equation}
R_{1}:=
\{\alpha\in R_\mathrm{nr}\mid
2\alpha\not\in R_\mathrm{nr}\}.
\end{equation}
Let $F_1\subset R_1$ be the bases of simple roots corresponding to $F_0$.
The root system $R_1$ differs from $R_0$ only if the root datum
of $\H$ contains direct summands of ``type $C_n^{(1)}$'', the irreducible
root datum with $R_0$ of type $B_n$ and $X$ the root lattice of $B_n$.
This is the only irreducible root datum for which the affine Hecke algebra
admits $3$ independent parameters. For this root datum, $R_1$ is of type
$C_n$. When applying Lusztig's first reduction theorem
(cf. \cite[Theorem 2.6]{OpSo})
one needs to consider the affine Weyl
group $R_1^{(1)}$ rather than $R_0^{(1)}$, and this is the reason
for introducing $R_1$ (it plays a role in the explicit
computations in Chapter \ref{sec:coc}).
\subsubsection{Restriction to parabolic subsystems}\label{subsub:par}
We define $\mathfrak{a}=Y\otimes_\mathbb{Z}\mathbb{R}$.
Let $P$ be a subset of $ F_0$. We have a canonical decomposition
$\mathfrak{a}=\mathfrak{a}^{P}\oplus\mathfrak{a}_P$, where
$\mathfrak{a}^P:=P^\perp$ and $\mathfrak{a}_P:=\mathbb{R}P^\vee$.
Dually we have the decomposition
$\mathfrak{a}^*=\mathfrak{a}^{P,*}\oplus\mathfrak{a}^*_P$ where
$\mathfrak{a}^*_P=\mathbb{R}P$ and $\mathfrak{a}^{P,*}=(P^\vee)^\perp$
(in the case $P=F_0$ we will denote this decomposition by
$\mathfrak{a}^*=\mathfrak{a}^{0,*}\oplus\mathfrak{a}^*_0$).
Let $R_P\subset R_0$ be the ``parabolic subsystem of roots''
$R_P=R_0\cap\mathfrak{a}^*_P$.

Consider the root datum $\Ri^P:=(X,Y,R_P,R_P^\vee,P)$.
Let $X_P\isom X/(X\cap\mathfrak{a}^{P,*})$ be the projection of the lattice $X$ on
$\mathfrak{a}^*_P$ along $\mathfrak{a}^{P,*}$ (this lattice contains the
lattice $X\cap\mathfrak{a}^*_P$ as a sublattice of finite index).
Observe that the dual lattice $Y_P$ of $X_P$ equals $Y_P=Y\cap\mathfrak{a}_P$.
We also introduce the semisimple root datum $\Ri_P:=(X_P,Y_P,R_P,R_P^\vee,P)$.
The non-reduced root systems associated to the root data $\Ri^P$ and $\Ri_P$
are both equal to $R_{P,\mathrm{nr}}:=\mathbb{Q}R_P\cap\rnr$. We define a
label function $q_P$ on the affine Weyl group associated to $\Ri_P$
by requiring that
the corresponding root label function on $R_{P,\mathrm{nr}}$ is obtained
by restricting the root label function on $\rnr$ to $R_{P,\mathrm{nr}}$.
We define a label function $q^P$ on the affine Weyl group
associated to $\Ri_P$ in the same fashion.
\subsubsection{Bernstein presentation}
There is a second presentation of
the algebra $\H$, due to Joseph Bernstein (unpublished).
Since the length function is additive on the dominant cone
$X^+$, the map $X^+\ni x\mapsto N_x$ is a homomorphism of the
commutative monoid $X^+$ with values in $\H^\times$, the group
of invertible elements of $\H$. Thus there exists a unique
extension to a homomorphism
$X\ni x\mapsto \theta_x\in\H^\times$
of the lattice $X$ with values in $\H^\times$.

Let $\A\subset\H$ be the abelian subalgebra of $\H$ generated by
$\theta_x$, $x\in X$.
Let $\H_0=\H(W_0,q_0)\subset\H$ be the finite type Hecke algebra
associated with $W_0$ and the restriction $q_0$ of $q$
to $W_0$. Then $\H_0\subset \H$ is a subalgebra of $\H$.
The Bernstein presentation asserts that
the multiplication maps $\H_0\otimes\A\to\H$ and $\A\otimes\H_0\to\H$
are linear isomorphisms. The algebra structure of $\H$ is then completely
determined by the following cross relations (for all $x\in X$ and
$s=s_\alpha$ with $\a\in F_0$):
\begin{gather}
\begin{split}
&\theta_xN_s-N_s\theta_{s(x)}=\\
&\left\{
\begin{array}{ccc}
&(q_{\alpha^\vee}^{1/2}-
q_{\alpha^\vee}^{-1/2})\frac{\theta_x-\theta_{s(x)}}
{1-\theta_{-\alpha}}&{\rm if}\ 2\alpha\not\in\rnr.\\
&((q_{\alpha^\vee/2}^{1/2}q_{\alpha^\vee}^{1/2}-
q_{\alpha^\vee/2}^{-1/2}q_{\alpha^\vee}^{-1/2})
+(q_{\alpha^\vee}^{1/2}-
q_{\alpha^\vee}^{-1/2})\theta_{-\alpha})
\frac{\theta_x-\theta_{s(x)}}
{1-\theta_{-2\alpha}}&{\rm if}\ 2\alpha\in\rnr.\\
\end{array}
\right.\\
\end{split}
\end{gather}
\subsubsection{The center $\Ze$ of $\H$}
An immediate consequence of the Bernstein
presentation of $\H$ is the description of the
center of $\H$:
\begin{thm}\label{thm:cent}
The center of $\H$ is equal to $\A^{W_0}$.
In particular, $\H$ is finitely generated
over its center.
\end{thm}
As an immediate consequence we see that irreducible
representations of $\H$ are finite dimensional by an
application of (Dixmier's version of) Schur's lemma.
\subsection{Intertwining elements}
Let $s=s_\alpha\in S_0$ with $\alpha\in F_1$. Define $\i_s\in\H$ by:
\begin{equation*}
\begin{split}
\i_s&=(1-\theta_{-\alpha})N_s+
((q_{\a^\vee}^{-1/2}q_{2\a^\vee}^{-1/2}-q_{\a^\vee}^{1/2}q_{2\a^\vee}^{1/2})
+(q_{2\a^\vee}^{-1/2}-q_{2\alpha^\vee}^{1/2})\theta_{-\alpha/2})\\
&=N_s(1-\theta_{\alpha})+
((q_{\a^\vee}^{-1/2}q_{2\a^\vee}^{-1/2}-q_{\a^\vee}^{1/2}q_{2\a^\vee}^{1/2})
\theta_{\alpha}
+(q_{2\a^\vee}^{-1/2}-q_{2\alpha^\vee}^{1/2})\theta_{\alpha/2})\\
\end{split}
\end{equation*}
(where, if $\a/2\not\in X$, we put $q_{2\a^\vee}=1$).
We recall from \cite[Theorem 2.8]{EO} that these elements of $\H$ satisfy the
braid relations, and they satisfy (for all $x\in X$):
\begin{equation}
\i_s\theta_x=\theta_{s(x)}\i_s
\end{equation}

Let $\mathcal{Q}$
denote the quotient field of the center $\Ze$ of $\H$, and let
${}_\mathcal{Q}\H$ denote the
$\mathcal{Q}$-algebra ${}_\mathcal{Q}\H=\mathcal{Q}\otimes_\Ze\H$.
Inside ${}_\mathcal{Q}\H$ we normalize the elements $\i_s$
as follows. We first introduce
\begin{equation}
n_\a:=q_{\alpha^\vee}^{1/2}q_{2\alpha^\vee}^{1/2}
(1+q_{\alpha^\vee}^{-1/2}\theta_{-\alpha/2})
(1-q_{\alpha^\vee}^{-1/2}q_{2\alpha^\vee}^{-1}\theta_{-\alpha/2})\in
\A.
\end{equation}Then the normalized intertwiners $\iota_s^0\in{}_\mathcal{Q}\H$
($s\in S_0$) are defined by (with $s=s_\a$,
$\a\in R_1$):
\begin{equation}\label{eq:defint}
\i^0_s :=n_\a^{-1}\i_s\in {}_\mathcal{Q}\H.
\end{equation}
It is known that $(\i_s^0)^2=1$, and in particular
that $\i_s^0\in{}_\mathcal{Q}\H^\times$, the
group of invertible elements of ${}_\mathcal{Q}\H$.
We have:
\begin{lemma}(\cite[Lemma 4.1]{O})
The map $S_0\ni s\mapsto \i^0_s\in {}_\mathcal{Q}\H^\times$ extends
(uniquely) to a homomorphism $W_0\ni w\mapsto
\i^0_w\in {}_\mathcal{Q}\H^\times$. Moreover,
for all $f\in {}_\mathcal{Q}\A$ we have that
$\i_w^0f\i_{w^{-1}}^0=f^w$.
\end{lemma}
\section{The Fourier transform for affine Hecke algebras}
Recall the canonical decomposition
$\mathfrak{a}^*=\mathfrak{a}^{0,*}\oplus\mathfrak{a}^*_0$.
Then $X\cap \mathfrak{a}^{0,*}$ consist of translations of length $0$
in the affine Weyl group. Choose a norm $\Vert\cdot\Vert$ on
$\mathfrak{a}^{0,*}$. Let us denote by $x^0$ the projection of
$x\in X$ onto $\mathfrak{a}^{0,*}$ along $\mathfrak{a}^*_0$.
Then we define a norm $\mathcal{N}$ on $W$ by
\begin{equation}
\Nc(w):=l(w)+\Vert w(0)^0 \Vert.
\end{equation}
\begin{dfn}
The Schwartz completion $\S$ of $\H$ is the
vector space of the formal complex linear
combinations $\sum_{w\in W}c_wN_w$ for which the function
$W\ni w\to c_w$ is rapidly decreasing with respect to the norm
$\Nc$ defined above on $W$, equipped
with the usual Fr\'echet topology on the space of rapidly
decreasing functions on $W$.
\end{dfn}
Recall the following result:
\begin{thm}(\cite[Theorem 6.5]{O})
The algebra structure on the dense subspace $\H\subset\S$ extends
uniquely to a Fr\'echet algebra structure on $\S$.
\end{thm}

Let us now review the notions involved in the definition of
two of the main ingredients involved in the description of the
structure of the Fr\'echet algebra $\S$, the groupoid $\W_{\Xi_u}$
of standard tempered induction data, and the ``induction intertwining
functor $\pi$'' defined on this groupoid. Both these structures
arise from the $L_2$-theory of the Hecke algebra.
\subsection{Tempered representations}
An affine Hecke algebra with a positive label function
comes equipped with the structure of a $*$-algebra, where $*$ denotes
the unique anti-linear anti-involution
defined by anti-linear extension of the map $N_w^*:=N_{w^{-1}}$.
Moreover, the linear functional $\tau$ defined by $\tau(N_w)=\delta_{w,e}$
is a positive trace with respect to $*$. The star operation $*$ and
trace $\tau$ together define a unique Hilbert algebra structure on $\H$
(see \cite{O}) which is the origin for the harmonic analysis on $\H$.

We define a positive definite Hermitian inner product on $\H$
by $(x,y):=\tau(x^*y)$, and denote by $L_2(\H)$ the Hilbert space
completion of $\H$. It is the separable Hilbert space in which the
basis elements $N_w$ ($w\in W$) form a Hilbert basis. We have
\begin{equation}
\H\subset\S\subset L_2(\H),
\end{equation}
and the second inclusion is easily seen to be continuous.

A representation $\pi$ of $\H$ which is of finite length is called
{\it tempered} if $\pi$ extends continuously to $\S$. It is in
fact sufficient that the character of $\pi$ (recall that all
representations of finite length of $\H$ are finite dimensional)
extends continuously to $\S$ (cf. \cite[Lemma 2.20]{O}).
An {\it irreducible} representation
$\pi$ of $\H$ is called a {\it discrete series} representation if
$\pi$ extends continuously to $L_2(\H)$. An equivalent way of
saying this is that the character $\chi_\pi$ of $\pi$ extends to
a continuous functional on $L_2(\H)$ (cf. \cite[Lemma 2.22]{O}).
Thus a discrete series module is in particular a tempered module.

Our main interest in this paper will be the description of the
structure of the tempered dual $\hat{\S}$ of $\H$, the set of
equivalence classes of irreducible tempered representations.
It is known (cf. \cite[Theorem 3.11, Theorem 3.19, Theorem 4.3]{DO}
and \cite[Theorem 2.25]{O}) that this set
of irreducible representations extends to the $C^*$-algebra
completion $C^*_r(\H)$ of $\H$, and that one obtains in this way
precisely the irreducible spectrum of $C^*_r(\H)$.
We equip $\hat{\S}$ with the topology of the spectrum of
$C^*_r(\H)$ via this identification.
\subsection{The groupoid of standard induction data}
\subsubsection{Induced representations}\label{par:inddat}
The affine Hecke algebra with root datum $\Ri^P$ and label function
$q^P$ is naturally embedded as a subalgebra of $\H$, as is apparent from
Bernstein's presentation.
The affine Hecke algebra with root datum $\Ri_P$ and label
function $q_P$ is isomorphic to the quotient of $\H^P$ by the central
subalgebra $\A^P\subset\H^P$ generated by the $\theta_x$ with
$x\in X\cap\mathfrak{a}^{P,*}$.
In particular for $\H_P$ the r\^ole of the complex
algebraic torus $T$ of (quasi) characters of $X$ is now played
by the algebraic subtorus $T_P\subset T$ with character lattice
$X_P$. The central characters of the irreducible
modules over $\H_P$ are $W_P$-orbits in $T_P$.

Let $T^P\subset T$ denote the complex algebraic subtorus
with character lattice $X^P=X/(X\cap\mathfrak{a}^*_P)$. This is the
identity component of the group of fixed points for the action
of $W_P$ on $T$. The group $T^P$ acts naturally on $\H^P$ by
automorphisms. We send $t\in T^P$ to the automorphism $\psi_t$ of
$\H^P$ which acts on the Bernstein basis by
$\theta_xN_w\to t(x)\theta_xN_w$. Given a
discrete series representation $\d$ of $\H_P$ we denote by $\d_t$
the twist $\d_t=\d\circ p\circ\psi_t$ by $t\in T^P$ of the lift of $\d$ to $\H^P$ via the
natural quotient map $p:\H^P\to\H_P$.
Define the finite abelian group
\begin{equation}\label{eq:KP}
K_P:=T^P\cap T_P\approx\textup{Hom}(X_P/(X\cap\mathfrak{a}^*_P),\mathbb{C}^\times)
\end{equation}
For later use we observe that if $k\in K_P$
then the twist by $k$ descends to an automorphism of $\H_P$, and we have
$\d_{tk}=(\d_k)_t$, where $\d_k(=\d^{k^{-1}})$ is the twist of $\d$
by the automorphism of $\H_P$ coming from $k$, which is again
a discrete series representation of $\H_P$.

Choose a complete set of representatives $\Delta_P=\Delta_{\Ri_P,q_P}$ for
the set of isomorphism classes of discrete series representations
of the Hecke algebra $\H_P$.
This set is finite \cite[Lemma 3.31]{O}.
We put $\Delta$ for the finite
disjoint union $\Delta:=\coprod_{P\in\P}(P,\Delta_P)$, a finite set with
a natural fibration $\Delta\to\P$
(with $\P$ the power set of $F_0$).

We will use some terminology from the theory of groupoids.
Recall that a groupoid
$\G$ is a ``group with several objects'' or more formally, a small
category in which all the morphisms are invertible. In particular a
group $G$ is a groupoid with one object. On the other extreme end any
set $X$ can be viewed as a groupoid with only identities, the
``identity groupoid'' of $X$.

A standard induction datum $\xi$ for $\H$ is a triple $(P,\d,t)$
with $P\in\P$, $\d\in\Delta_P:=\Delta_{\Ri_P,q_P}$, and $t\in T^P$.
Recall that
$\d$ is a representative of an equivalence class of discrete
series representations. Let us denote the underlying vector space
by $V_\d$.
The set $\Xi$ of all such triples is a finite (by
\cite[Lemma 3.31]{O}) disjoint union of the subsets $\Xi_{(P,\d)}$,
each of which is a copy of the complex algebraic torus $T^P$.
We view $\Xi$ as the set of arrows of a
groupoid whose set of objects is $\Delta$, with
$\operatorname{Hom}((P,\d),(Q,\tau))=\Xi_{(P,\d)}$ if $(P,\d)=(Q,\tau)$
and $=\emptyset$ else. We identify $\Xi_{(P,\d)}$ with the complex
algebraic torus $T^P$ by $T^P\ni t\to\xi=(P,\d,t)\in\Xi_{(P,\d)}$.
This equips $\Xi$ in particular with the structure of a complex
algebraic variety. We denote by $\Xi_u\subset\Xi$ the compact real
form of $\Xi$ (i.e. we restrict $t$ to the compact real form
$T^P_u\subset T^P$). Given an
induction datum $\xi=(P,\d,t)$ we can define an {\it induced
representation}
$\pi(\xi)$ of $\H$ by inducing $\d_t$ from $\H^P$ to $\H$
(see \cite[Paragraph 4.5.1]{O}; \cite[Subsection 3.5]{DO}).
The representation is realized in the vector space
\begin{equation}
V_\xi=\H\otimes_{\H^P}V_\d=:i(V_\d)
\end{equation}
which is {\it independent of $t\in T^P$} (the ``compact realization'').
The matrix coefficients of $\pi(\xi)$ are regular functions
on $\Xi$. The representations $(\pi(\xi),V_\xi)$ are called
{\it generalized principal series} representations.
\begin{prop}(\cite[Proposition 4.19, Proposition 4.20]{O})
The generalized principal series $\pi(\xi)$ is tempered if
$\xi\in\Xi_u$, and it is unitary for $\xi\in \Xi_u$ with
respect to a standard inner product on $i(V_\d)$ which is
independent of $\xi$.
\end{prop}
\subsubsection{The groupoid of standard induction
data}\label{subsub:groupoid}
We now describe the morphisms of standard induction data.
Recall the Weyl groupoid $\Wf$ which has the collection of
standard parabolic subsets $P\subset F_0$ as set of objects, with
arrows $\Wf_{P,Q}:=\{w\in\Wf|w(P)=Q\}$ between two standard
parabolic subsets $P,Q$ (see the Appendix
\ref{app:weylgroupoid} for some important notions related to
$\Wf$).
If $w\in\Wf_{P,Q}$ then there exists a corresponding isomorphism
of root data $\Ri_P\to\Ri_Q$ compatible with the root labels $q_P$
and $q_Q$, thus defining an isomorphism $\psi_w:\H_P\to\H_Q$. On
the other hand, we have already seen above that with $k\in K_Q$
there is associated a twist $\psi_k$ of $\H_Q$. We define
a groupoid $\W$ whose set of objects is $\P$ and
$\W_{P,Q}:=K_Q\times\Wf_{P,Q}$ with the obvious composition rule
$(k\times u)\circ(l\times v)=k(u(l))\times uv$. We also introduce
the ``normal subgroupoid'' $\mathcal{K}\subset \mathcal{W}$ whose
objects are $\mathcal{P}$, with $\mathcal{K}_{P,Q}=\emptyset$ if
$P\not=Q$ and $\mathcal{K}_{P,P}=K_P$. Hence the Weyl groupoid
$\Wf$ is equal to the quotient $\Wf=\W/\K$.

For each $g=k\times u\in\W$ we define an
isomorphism $\psi_{g}:\H_P\to\H_Q$ by $\psi_g=\psi_{k\times
u}:=\psi_k\circ\psi_u$.
There is a natural action of $\W$ on the space $\Xi$ by (with
$g=k\times u\in K_Q\times\Wf_{P,Q}$):
\begin{equation}\label{eq:actgro}
g(P,\d,t):=(u(P),\d^g,g(t)),
\end{equation}
where $\d^g\in\Delta_Q$ is the unique discrete series
representation such that $\d^g\simeq\d\circ\psi_g^{-1}$.
\begin{dfn}
We define the {\it groupoid $\W_\Xi$ of
standard induction data} by $\W_\Xi:=\W\times_\P\Xi$. Its set of
objects is $\Xi$, and the morphisms in $\W_\Xi$ from $\xi\to\eta$
are the $g\in\W$ such that $g(\xi)=\eta$. The full subgroupoid
$\W_{\Xi_u}$ is obtained by restricting the set of objects to
$\Xi_u\subset\Xi$.
\end{dfn}
Notice that the groupoid $\W_\Xi$ is canonically determined
by $\H$. In particular $\W_\Xi$ is independent of the chosen
representatives of the isomorphism classes of discrete series
representations $\Delta$.

Given a morphism $g=k\times u\in\W_{P,Q}$ in the groupoid
and a discrete series representation $\d\in \Delta_P$,
we now {\it choose} an isomorphism
\begin{equation}\label{eq:choice}
\tilde\d_g:V_\d\to V_{\d^g}
\end{equation}
intertwining the irreducible
representations $\d\circ\psi_g^{-1}$ and $\d^g$.
Given $\xi=(P,\d,t)\in\Xi_(P,\d)$
we will define a normalized intertwining operator
\begin{equation}\label{eq:normintdefn}
\pi(g,\xi):(\pi(\xi),i(V_\d))\to (\pi(g\xi),i(V_{\d^g}))
\end{equation}
under certain regularity conditions on $\xi$ (see the discussion below;
for further detail we refer to\cite[Section 4.4]{O} and to
\cite[equation (3.8)]{DO}).
The definition is complicated since it involves the
intertwining elements $\i_{u^{-1}}^0\in{}_\mathcal{Q}\H$,
which act in a representation $\pi(\xi)$ only if $\xi$ is such
that the poles of the intertwining elements are avoided.
In \cite[Section 4.4]{O} the normalized intertwining operators $\pi(g,\xi)$
are first defined algebraically in the Zariski open
set of the so-called $R_P$-generic elements $\xi\in\Xi_{(P,\d)}$, and afterwards
extended to a larger open set in the analytic topology containing
$\Xi_{(P,\d),u}$, using the unitarity of the $\pi(g,\xi)$.
The element $\xi=(P,\delta,t)\in\Xi_{(P,\d),u}$ is called $R_P$-generic if the orbit
$W_Pr_\delta t\subset T$ consists of $R_P$-generic elements in
the sense of
\cite[Definition 4.12]{O}, where $W_Pr_\delta\subset T_P$ denotes the
central character of $\delta$. The set of such $R_P$-generic $\xi$ is
Zariski-open in $\Xi_{(P,\d)}$, and for $\xi$ in this set we can
define the normalized intertwining operator on $\pi(\xi)$ by the formula
\begin{equation}\label{eq:normint2}
\pi(g,\xi)(N_w\otimes v)=
\pi(u(\xi),N_w)\pi(u(\xi),\i_{u^{-1}}^0)
(1\otimes\tilde\d_g(v)),
\end{equation}
(see \cite[Section 4.4]{O}, or \cite[Section 3.5]{DO}).
However, in a suitable open neighborhood of $\Xi_{(P,\d),u}$ the apparent poles
of the normalized intertwining operators turn out to be removable
(see \cite[subsection 4.4]{O}). Hence we can uniquely extend the $\pi(g,\xi)$
to a smooth family of operators depending on $\xi$ in a suitable open
neighborhood of $\Xi_{(P,\d),u}$.
The normalized intertwining operators on $\Xi_u$ we use in the present
paper are the restrictions to $\Xi_u$ of
these regular rational functions of $\xi$ defined in a neighbourhood of
$\Xi_u\subset \Xi$. They
are in fact unitary for $\xi\in \Xi_u$ with respect to the standard
inner products on $i(V_\xi)$ and $i(V_{g(\xi)})$ (cf. \cite[Proposition 4.19]{O}).
In this way we have obtained the following result
(see \cite[Theorem 3.38]{O} and \cite[Theorem 3.14]{DO}):
\begin{thm}\label{thm:indint}
The assignment $\Xi_{(P,\d),u}\ni\xi\to\pi(\xi)$ and
$\W_{P,Q}\ni g\to\pi(g,\xi)$ extends to a functor $\pi$ (the
``induction intertwining'' functor) from
$\W_{\Xi_u}$ to $\mathbb{P}\operatorname{Rep}(\H)_{temp,unit}$, the
category of tempered, unitary modules of $\H$ in which the
morphisms are unitary $\H$-intertwiners modulo scalars. The
functor $\pi$ is rational and regular in $\xi\in\Xi_u$.
\end{thm}
\begin{thm}(\cite[Theorem 3.19 and Corollary 5.6]{DO})
Any irreducible tempered representation $V$ is isomorphic to a summand
of a generalized principal series representation for a unitary standard
induction datum $\xi\in\Xi_u$ whose isomorphism class is uniquely
determined by $V$.
\end{thm}
This theorem tells us that in order to classify the irreducible
tempered representations, it is enough to classify the discrete
series representations and to understand how the generalized
principal series representations with unitary induction parameter
decompose in irreducible subrepresentations. The theory of the
analytic $R$-group below is designed to resolve this last problem of the
decomposition of $\pi(\xi)$ for unitary $\xi$.
\subsubsection{The 2-cocycle $\ga_{\W,\Delta}$}\label{subsub:2coc}
It is
conventional to denote by $\G^0$ the set of objects of a groupoid
$\G$, and
by $\G^1$ the set of morphisms or arrows of $\G$. Each arrow
$g\in\G^1$ has a source object $s(g)$ and a target object $t(g)$,
and this defines two maps $s,t:\G^1\to\G^0$. The set of composable
pairs of arrows is $\G^2:=\{(g_1,g_2)\mid g_i\in\G^1,
s(g_2)=t(g_1)\}$. This set is thus is a fibered product
\begin{equation}
\G^2=\G^1{}_s\times_t\G^1.
\end{equation}

The twisting isomorphisms $\psi_g:\H_P\to\H_Q$ with $g\in\W_{P,Q}$ define
a homomorphism from the groupoid $\W$ to the groupoid
$\operatorname{Iso}_\P$ whose set of objects is $\P$ and whose morphisms
$\operatorname{Iso}_\P(P,Q)$ consist of algebra isomorphisms from $\H_P$ to $\H_Q$
which map $\A_{P,+}$ to $\A_{Q,+}$, where $\A_{P,+}$ is the subalgebra
of $\H_P$ spanned by the elements $N_x$ with $x\in X_{P,+}$ and
similarly for $\A_{Q,+}$. This induces an action of $\W$ on the
set $\Delta$. The choice of intertwining isomorphisms made in (\ref{eq:choice})
determines a $\operatorname{U}(1)$-valued
$2$-cocycle $\ga_{\W,\Delta}$ on the finite groupoid $\W_\Delta=\W\times_\P\Delta$
(which has the finite set $\Delta$ as its set of objects) as follows. Let
$g^\prime,g$ be composable arrows in $\W$, let
$\delta\in\Delta_P$ where $P$ is the source of $g$,
and let $\d^\prime=\d^g$.
With the above notations, we have
$((g^\prime,\d^\prime),(g,\d))\in
\W_\Delta^2:=\W_\Delta {{}_s\times_t}\W_\Delta$
(where $s,t$ denote the
source and target map of the groupoid $\W_\Delta$).
We define a function $\ga_{\W,\Delta}$ on $\W_\Delta^2$ by:
\begin{equation}\label{eq:coc}
{\tilde\d}^\prime_{g^\prime}\circ{\tilde\d}_g=
\ga_{\W,\Delta}((g^\prime,\d^\prime),(g,\d))\tilde\d_{g^\prime g}.
\end{equation}
\begin{prop}\label{prop:coc}
The function $\ga_{\W,\Delta}$ defines a $2$-cocycle on
$\W_\Delta^2$ with values in
$\operatorname{U}(1)$, whose class $[\ga_{\W,\Delta}]\in H^2(\W_\Delta,\operatorname{U}(1))$
is independent of the choices of the $\tilde\d_g$. We can choose the $\tilde\d_g$ such that
$\ga_{\W,\Delta}$ has values in the group $\mu(D_\Delta)$ of complex $D_\Delta$-th roots of
unity, with
$D_\Delta:=\operatorname{lcm}_{\d\in\Delta}\{\operatorname{dim}(V_\d)\}$.
\end{prop}
\begin{proof}
Let $((f,\d),(g,\e),(h,\zeta))\in\W^3_\Delta:=
\W_\Delta{{}_s\times_t}\W_\Delta{{}_s\times_t}\W_\Delta$.
By associativity
\begin{equation}
(\tilde{\d}_f\circ\tilde{\e}_g)\circ\tilde{\zeta}_h=
\tilde{\d}_f\circ(\tilde{\e}_g\circ\tilde{\zeta}_h)
\end{equation}
one checks the $2$-cocycle relation of $\ga_{\W,\Delta}$.
It is clear that changing the choices of the isomorphisms $\tilde\d_g$
in equation (\ref{eq:choice}) changes $\ga_{\W,\Delta}$ only by a coboundary.
Let us now prove the last assertion. First suppose that $D_\Delta=1$.
Choose a basis vector in $V_\delta$ for each pair $(P,\delta)$ with
$\d\in\Delta_P$. Then $\tilde{\d}_g$ is a complex scalar, and relation (\ref{eq:coc})
expresses $\ga_{\W,\Delta}$ as the coboundary of the $\mathbb{C}^\times$-valued
function $g\to\tilde{\d}_g$ on $\W_\Delta^1$, proving the assertion in this special
case. In the general case, taking determinants (and suitable powers) in (\ref{eq:coc})
similarly shows that $\ga_{\W,\Delta}^{D_\Delta}$ is the coboundary of a
$\mathbb{C}^\times$-valued function $\epsilon$ on $\W_\Delta^1$, i.e.
$\ga_{\W,\Delta}(g,g^\prime)^{D_\Delta}=
\epsilon(g)\epsilon(g^\prime)\epsilon(g\circ g^\prime)^{-1}$ for all
$(g,g^\prime)\in\W_\Delta^2$.
Now choose $\zeta(g)$ for $g\in\W_\Delta$ such that $\epsilon(g)=\zeta^{D_\Delta}(g)$,
and replace $\tilde{\d}_g$ by $\tilde{\d}_g^\prime=\zeta^{-1}\tilde{\d}_g$.
The $2$-cocycle $\ga^\prime_{\W,\Delta}$ defined by (\ref{eq:coc})
after replacing $\tilde{\d}_g$ by $\tilde{\d}_g^\prime$ takes values in $\mu(D_\Delta)$.
\end{proof}
The projective representation $\pi$ of the groupoid $\W_{\Xi_u}$
is related to $\ga_{\W,\Delta}$ by the following formula, which follows
immediately from the definition of the cocycle $\ga_{\W,\Delta}$, the definition
of the normalized intertwining operators, and from the fact that the
normalized intertwining elements satisfy the Weyl group relations
(cf. \cite[Lemma 4.1]{O}): Let $\xi=(P,\d,t)\in\Xi_{(P,\d),u}$, then
\begin{equation}\label{eq:groupoidcoc}
\pi(h,g\xi)\circ\pi(g,\xi)=\ga_{\W,\Delta}((h,\d^g),(g,\d))\pi(hg,\xi)
\end{equation}
\begin{dfn}\label{dfn:gaxi}
Formula (\ref{eq:groupoidcoc}) defines a torsion $2$-cohomology class
$[\ga_{\W,\Xi}]\in H^2(\W_{\Xi_u},\operatorname{U}(1))$, the pull back of $[\ga_{\W,\Delta}]$
via the natural homomorphism of groupoids $\W_{\Xi_u}\to\W_{\Delta}$.
\end{dfn}
One should think of
$[\ga_{\W,\Xi}]$ as the characteristic class of the projective
bundle over the groupoid $\W_\Xi$ which is defined by the
induction intertwining functor $\pi$.
\subsubsection{Inertial orbits of discrete series modulo the center}\label{subsub:O}
It is sometimes convenient to work with a slightly modified version
of the groupoid of standard induction data.

We say that an irreducible representation $\si$ of $\H^P$ is discrete series
modulo the center if $\si$ is equivalent to a representation of the form
$\delta_t$ with $t\in T^P_u$ and $\delta$ an irreducible discrete series
representation of $\H_P$ (see paragraph \ref{par:inddat}).
Let $P\subset F_0$ and let $\si$ be a discrete series representation modulo the
center of $\H^P$.
By definition the inertial orbit $\O_{(P,\si)}$ of $\si$ is the set of
the equivalence classes $(P,\si\circ\psi_t)$ of the $\H^P$-representations
$\si\circ\psi_t$ where $t$ varies in $T^P$.
This gives a natural $T^P$ action on $\O_{(P,\delta)}$.

It is obvious that the
isotropy of a datum $(P,\si)\in\O_{(P,\delta)}$ is always a finite
subgroup of $T^P$.
It is also clear (see also the discussion in paragraph
\ref{par:inddat}) that each orbit $\O_{(P,\si)}$ contains
a unique $K_P$-orbit of discrete series modulo the center which
descend to $\H_P$.
If $\si_P$ descends to $\H_P$ then there exists a discrete series
representation $\d$ of $\H_P$ such that $[\d_1]=\si_P$.
Therefore there exists, for each orbit $\O_{(P,\si)}$, a component
$\Xi_{(P,\d)}$ of $\Xi$ and a finite covering map
\begin{align}\label{eq:cov}
\Xi_{(P,\d)}&\to \O_{(P,\si)}\\
\nonumber(P,\d,t)&\to (P,[\d_t])
\end{align}
In this case we will also use the notation $\O_{(P,\d)}$ to
denote $\O_{(P,\si)}$.
It is easy to see that for given
$\xi=(P,\d,t)$ and $\xi^\prime=(P,\d^\prime,s)$
we have $[\d^\prime_s]=[\d_t]$ (isomorphic as representations of
$\H^P$) if and only if $K_P\xi=K_P\xi^\prime$. In other words, we
have
\begin{equation}\label{eq:covmap}
\O=\K\backslash\Xi=|\K_\Xi|,
\end{equation}
(where $|\K_\Xi|$ denotes the orbit space of isomorphism classes
of objects of the groupoid $\K_\Xi=\K\times_\P\Xi$), and the
covering (\ref{eq:cov}) is given by taking the
quotient of $\Xi_{(P,\d)}$ by the isotropy subgroup
$K_\d\in K_P$ of $[\d]$. Since the action of $\K$ on $\Xi$ is
free, the orbit map extends to a homomorphism of groupoids
(viewing $\O$ as the ``unit'' groupoid with only identity
morphisms)
\begin{equation}\label{eq:OKmorita}
\K_\Xi\to\O
\end{equation}
which is a Morita equivalence (in the sense of \cite{Moe}).
The space $\O$ is a disjoint union of finitely many orbits
of the form $\O_{(P,\d)}$ (parameterized by the $\K$-orbits
on $\Delta$), and each orbit $\O_{(P,\d)}$ has the natural
structure of a $T^P/{K_\d}$-torsor (corresponding to the multiplication
action of $T^P$ on $\Xi_{(P,\d)}$ by identifying $\Xi_{(P,\d)}$
with $T^P$). This gives $\O$ the structure of
a complex algebraic variety and it defines a special compact
form $\O_u$ of $\O$.

Clearly $\O$ carries a natural action of the Weyl groupoid $\Wf=\W/\K$.
We consider the groupoid
\begin{equation}
\Wf_\O:=\Wf\times_\P\O
\end{equation}
(and its compact form $\W_{\O_u}$).
The observations made in this paragraph amount to saying that:
\begin{prop}\label{prop:morita}
The groupoids $\W_\Xi$ ($\W_{\Xi_u}$)
and $\Wf_\O$ (resp. $\Wf_{\O_u}$) are Morita equivalent.
\end{prop}
However, it is important to observe at this point that:
\begin{rem}\label{rem:notmorita}
Let $|\Delta|$ denote the set of isomorphism classes of the
normal subgroupoid $\K_\Delta$ of $\W_\Delta$.
The quotient homomorphism $\W_\Delta\to\Wf_{|\Delta|}$ (defined by
sending $w\times k\to w$ and $\d\to\K\d$) is a Morita equivalence
if and only if all the isotropy groups $K_\d$ are trivial.
\end{rem}
\subsection{The Fourier isomorphism}
We will formulate the main result of \cite{DO} (see loc.cit
Section 5) in this section. Denote by
the trivial vector bundle over $\Xi$ whose fibre at $\xi$ is
equal to $V_\xi=i(V_\d)$, thus
\begin{equation}
\V_\Xi:=\coprod_{(P,\d)}\Xi_{P,\d}\times i(V_\d)
\end{equation}
The algebra of smooth sections of the trivial bundle
$\operatorname{End}(\V_{\Xi})$ on $\Xi_u$ will be denoted by
$\cc(\Xi_u,\operatorname{End}(\V_{\Xi}))$\index{C@$\cc(\Xi_u,
\operatorname{End}(\V_{\Xi}))$,
algebra of smooth sections in $\operatorname{End}(\V_{\Xi})$}.
We equip this algebra with its usual Fr\'echet topology.
We define the set of $\W$-equivariant sections in this bundle
as follows. Recall that $\pi(g,\xi)$
is smooth and has smooth inverse on $\Xi_u$.
Take $\xi\in\Xi_{P,u}$ and let
$A$ be an element of $ \operatorname{End}(V_\xi)$.
For $g\in\W_\xi$ (where $W_\xi$ denotes the set of elements in
$\W$ which act on $\xi$, hence with source $P$) we define
$g(A):=\pi(g,\xi)\circ A\circ \pi(g,\xi)^{-1}\in
\operatorname{End}(V_{g(\xi)})$.
\begin{dfn} A section of
$f$ of $\operatorname{End}(\V_\Xi))$ is called
$\W$-equivariant if we have $f(\xi)=g^{-1}(f(g(\xi)))$
for all $\xi\in \Xi$ and $g\in\W_\xi$. We denote the
subalgebra of smooth $\W$-equivariant sections by
$\cc(\Xi_u,\operatorname{End}(\V_{\Xi}))^\W$.
\end{dfn}
The Fourier transform $\F$ is canonically defined in terms of the
induction intertwining functor $\pi$: Given $x\in\S$ we define a
section $\F(x)$ of $\operatorname{End}(\V_{\Xi})$ by
$\F(x)(\xi):=\pi(\xi,x)$. The fact that the target of $\pi$ is
a category whose objects are unitary representations of $\H$ implies
that $\F$ is an algebra homomorphism, and the functoriality of $\pi$
amounts to the fact that $\F(x)$ is a $\W$-equivariant section in the
above sense. In \cite[Proposition 7.3]{DO} it was shown that in fact
$\F(\S)\subset\cc(\Xi_u,\operatorname{End}(\V_{\Xi}))^\W$
(this inclusion is not very hard to prove).

We define a wave packet operator at first as the isometry
\begin{equation}
\J:L_2(\Xi_u,\operatorname{End}(\V_{\Xi}),\mu_{Pl})\to L^2(\H)
\end{equation}
(where $\mu_{Pl}$ is the Plancherel measure, cf. \cite[Section 4]{DO})
which is the adjoint of the $L_2$-extension of the
Fourier transform. From the expression of the density
function of the Plancherel measure it is
easy to see that the space
\begin{equation}
\mathcal{C}(\Xi_u,\operatorname{End}(\V_{\Xi})):=c
\cc(\Xi_u,\operatorname{End}(\V_{\Xi})),
\end{equation}
where $c$ denotes the $c$-function on $\Xi_u$
(see e.g. \cite[Definition (9.7)]{DO}), is
a subspace of the Hilbert space
$L_2(\Xi_u,\operatorname{End}(\V_{\Xi}),\mu_{Pl})$.
Hence $\J$ is well defined on this vector space.
We equip $\mathcal{C}(\Xi_u,\operatorname{End}(\V_{\Xi}))$
with the Fr\'echet topology of
$\cc(\Xi_u,\operatorname{End}(\V_{\Xi}))$ via the
linear isomorphism
$\cc(\Xi_u,\operatorname{End}(\V_{\Xi}))\to
\mathcal{C}(\Xi_u,\operatorname{End}(\V_{\Xi}))$
defined by $\sigma\to c\sigma$.
Finally we define an averaging projection
$p_\W$ onto the space of $\W$-equivariant sections
by:
\begin{equation}
p_\W(f)(\xi):=|\W_\xi|^{-1}
\sum_{g\in \W_\xi}g^{-1}(f(g(\xi))).
\end{equation}
We can now formulate the main result of \cite{DO}:
\begin{thm}\label{thm:mainDO}
The Fourier transform restricts to an isomorphism of
Fr\'echet algebras
\begin{equation}
\F:\S\to C^\infty(\Xi_u,\operatorname{End}(\V_\Xi))^\W.
\end{equation}
The wave packet operator $\J$ restricts to a surjective
continuous map
\begin{equation}\label{eq:js}
\J_{\mathcal{C}}:\mathcal{C}(\Xi_u,\operatorname{End}(\V_\Xi))\to\S.
\end{equation}\index{J@$\J$, wave packet operator,
adjoint of $\F$!$\J_{\mathcal{C}}$, restriction of
$\J$ to the Fr\'echet space
$\mathcal{C}(\Xi_u,\operatorname{End}(\V_\Xi))$}
We have $\J_{\mathcal{C}}\F=\operatorname{id_\S}$, and we have
$\F\J_{\mathcal{C}}=
p_\W|_{\mathcal{C}(\Xi_u,\operatorname{End}(\V_\Xi))}$.
In particular, the map $p_\W$ is a continuous
projection of
$\mathcal{C}(\Xi_u,\operatorname{End}(\V_{\Xi}))$ onto
$\cc(\Xi_u,\operatorname{End}(\V_{\Xi}))^\W$.
\end{thm}
The projection $p_\W$ thus cancels singularities of sections over
$\Xi_u$ which are no worse than the poles
of the $c$-function on $\Xi_u$. This property of $p_\W$ is crucially
important in the sequel of the paper.
\section{The analytic $R$-group}\label{sec:anR}
In this section we will define the notion of the analytic
$R$-group $\mathfrak{R}_\xi$ in our context for a given unitary standard
induction datum $\xi\in\Xi_u$. Our treatment follows closely the argument of
\cite{S} but is more direct. For a good account of the r\^ole of
the $R$-group in the work of Harish-Chandra and of Knapp and Stein
\cite{KS} we refer the reader to \cite[Section 2]{A}.

The group $\mathfrak{R}_\xi$ is a subgroup of the inertia
group $\W_{\xi,\xi}$ which is a complement of a certain normal
reflection subgroup $\W^m_{\xi,\xi}$ of $\W_{\xi,\xi}$. The
reflection hyperplanes of the reflections in $\W^m_{\xi,\xi}$ are described
in terms of the Plancherel density function. The importance of the
$R$-group $\mathfrak{R}_\xi$ is that the induced module $\pi(\xi)$ (which
naturally comes with the structure of a
$\H-{}^\ga\mathbb{C}[\W_{\xi,\xi}]^{op}$ bimodule via the
induction-intertwining functor) is a Morita equivalence module
between the opposite of the $\ga_\xi$ twisted group ring of $\mathfrak{R}_\xi$
(for a certain $2$-cocycle $\ga_\xi$ derived from $\ga_{\W,\Xi}$) on the one
hand, and the category of tempered unitary $\H$-modules with central character
$\W\xi$ (in the sense of a character of the center of the Schwartz algebras $\S$,
see \cite[Corollary 5.5]{DO}) on the other hand. This implies in
particular that the irreducible tempered modules of $\H$ with
central character $\W\xi$ are in one-to-one correspondence with the
irreducible characters of ${}^\ga\mathbb{C}[\mathfrak{R_\xi}]$
(see \cite[Section 2]{A}).
\subsubsection{Definition of the $R$-group}
We identify $\Xi_{(P,\d)}$ with the complex torus $T^P$, and in
doing so, we in particular give meaning to group theoretical operations
in $\Xi_{(P,\d)}$ (such as $\xi^{-1}$). Below we use notations and concepts
associated to the chamber system of the Weyl groupoid and restrictions
of roots to facets of the Weyl chamber; we refer the reader to Section
\ref{app:weylgroupoid} for these notations and some basic facts.
We also recall the decomposition
$\mathfrak{a}=\mathfrak{a}^P\oplus\mathfrak{a}_P$
(see paragraph \ref{subsub:par}) for $P\in\P$.

The rational function $\nu(\xi)=(c(\xi)c(\xi^{-1}))^{-1}$
(which is the density function for the Plancherel measure, up to
normalizing constants) is known to be regular and positive on $\Xi_u$
(cf. \cite[Proposition 9.8]{DO}). This applies to the corank $1$ factors
of the $c$-function as well, so by the product formula for
$\nu$ (see \cite[Definition 9.7]{DO}) it is clear that the
zero set of $\nu$ in $\Xi_{(P,\d),u}$ (with $(P,\d)\in\Delta$)
is a finite union of orbits $M_{(P,\a),\xi,u}$
(with $(P,\a)\in R^P$ and $\xi\in \Xi_{(P,\d),u}$) of codimension $1$
subtori of the form $T^{(P,\a)}_u\subset T_u^P$, the unique
codimension one subtorus which lies in the kernel of the character
$(P,\a)\in R^P$.
\begin{dfn}
The orbits of the form $M_{(P,\a),\xi,u}$ in the zero set of $\nu$ intersected
with $\Xi_u$ are called mirrors in $\Xi_u$. The collection of all mirrors is
denoted by $\M$. The set of mirrors in $\Xi_{(P,\d),u}$ is denoted
by $\M_{(P,\d)}$, so that $\M=\coprod_{(P,\d)\in\Delta}\M_{(P,\d)}$ (a disjoint union).
\end{dfn}
\begin{prop}\label{prop:Minv}
The collection $\M$ is $\W$-invariant.
\end{prop}
\begin{proof}
This is clear by the $\W$-invariance of $\nu$.
\end{proof}
The next theorem is inspired by well known results of
Harish-Chandra (see \cite[Section 39]{HC3},
and also \cite{KS}, \cite{S}):
\begin{thm}\label{thm:mirror}
Let $M\in\M_{(P,\d)}$.
\begin{enumerate}
\item
There exists a unique
involution $\s_M\in\W_{(P,\d),(P,\d)}$ (the inertia group in
$\W_\Delta$ of the object $(P,\d)$) such that
$\s_M$ leaves $M$ pointwise fixed.
\item
$\s_M$ is $\W$-conjugate to an element of the form
$\s^\prime\times k^\prime$ with $\s^\prime=\s_{Q^\prime}^{P^\prime}\in
\Wf_{P^\prime,P^\prime}$ an elementary conjugation
with $P^\prime\subset Q^\prime$ self-opposed
(see paragraph \ref{subsub:eltconj}), and $k^\prime\in
K_{P^\prime}$ such that
$\s^\prime\times k^\prime=(k^\prime)^{-1}\times \s^\prime$.
\item
The rational function $\xi\to c(\xi)$ on $\Xi_{(P,\d),u}$ has a
pole of order one at $M$.
\item For all $\xi\in M$, the intertwining operator $\pi(\s_M,\xi)$
is a scalar.
\end{enumerate}
The element $\s_M$ is called the reflection in $M$.
\end{thm}
\begin{proof}
The proof of this result is based on the following aspect of the main
Theorem 5.3 of \cite{DO}. Let $V=i(V_\d)$ denote the vector space
on which all the induced representations $\pi(\xi)$ (with
$\xi\in\Xi_{(P,\d)}$) are realized in the compact realization. Let
\begin{equation}
f:\Xi_{(P,\d),u}\to\operatorname{End}(V)
\end{equation}
be a smooth section, and extend this function by $0$ on the other
components of $\Xi_u$. Then Theorem 5.3 of \cite{DO} implies that the
function $p_\W(cf)$ on $\Xi_u$ (where $c$ denotes the $c$-function on
$\Xi_u$) defined by
\begin{equation}\label{eq:av}
p_\W(cf)(\xi)=|\W_\xi|^{-1}
\sum_{g\in \W_\xi}\pi(g,\xi)^{-1}(cf(g(\xi)))\pi(g,\xi)
\end{equation}
is again smooth on $\Xi_u$.

Recall that (by
the Maass Selberg relations, see \cite[Proposition 9.8]{DO}) the
function $c$ vanishes on $M$ since $\nu$ vanishes on $M$.
Let $\xi\in M$ and let $\W_{\xi,\xi}\subset\W_{(P,\d),(P,\d)}$
denote the subgroup of elements which fix $\xi$. If the identity is
the only element of $\W$ which fixes the elements of $M$ pointwise then
$\W_{\xi,\xi}=\{e\}$ for generic $\xi\in M$. In that case
there would exist
a small open neighborhood $U\ni\xi$ such that $wU\cap U=\emptyset$
if $w\in\W_\xi$ but $w\not=e$. Hence if we take $f$ such that its support
is contained in $U$ but with $f(\xi)\not= 0$ the expression
(\ref{eq:av}) will not be smooth on $U$, a contradiction.

We conclude that there exists an element $\s\in\W_{(P,\d),(P,\d)}$ which fixes
$M$ pointwise and which is not the identity on $\Xi_{(P,\d)}$.
Thus locally in the tangent space $i\mathfrak{a}^P$ of $\Xi_{(P,\d)}$ at $\xi\in M$,
$\s$ must be given by an element of $\Wf_{P,P}$ that  fixes the hyperplane in
$i\mathfrak{a}^P$ which corresponds to $M$ under the exponential mapping
(where we choose $\xi$ as the identity element of $\Xi_{(P,\d)}$).
This uniquely
determines $\s$ on $\Xi_{(P,\d),u}$ and shows that $\s$ is an
involution. It also follows that $\s$ is $\W$-conjugate to an
involutive elementary conjugation (see paragraph \ref{subsub:eltconj})
composed with an element of $\K$
such that the composition is still an involution, proving both
(i) and (ii).

Let us now consider (iii).  Take a generic element $\xi=(P,\d,t_0)$
of $M$ such that $\W_{\xi,\xi}=\{e,\s_M\}$ and let $U$ be
a small open neighborhood of $\xi$ which is invariant for $\s_M$ and
has the property that $wU\cap U\not=\emptyset$ iff $w(\xi)=\xi$.
For $t\in T^P_u$ we write $\xi_t=(P,\d,tt_0)$.
By (ii), there exists a unique pair $(P,\a), (P,-\a)\in R^P$ of opposite
roots such that the function $\chi_\a:\Xi_{(P,\d)}\to\C^\times$ defined by
$\chi_\a(\xi_t):=\a(t)-1$ has the property that
$M\cap U=\{\xi\mid\chi(\xi)=0\}\cap U$.
Observe that $\chi_\a(\s_M(\xi_t))=-\a(t)^{-1}\chi_\a(\xi_t)$. Let
$U_T\subset T_u^P$ denote the open neighborhood of $e\in T^P_u$ such
that $\xi_t\in U$ iff $t\in U_T$.

Suppose now that the order of the pole of $c$ at $M$ is larger than $1$.
Then for an arbitrary smooth section $f$ with support in $U$ as
before, there exists a smooth section $h$ with support in $U$ such
that $\chi_\a^{-2}f=ch$.
Hence $p_\W(\chi_\a^{-2}f)$ is smooth by Theorem 5.3 of \cite{DO}
(see (\ref{eq:av})). In view of the choice of $U$ and
(\ref{eq:av}) this implies that the expression
\begin{equation}\label{eq:ord2}
\chi_\a^{-2}(\xi_t)(f(\xi_t)+
\a(t)^{2}\pi(\s_M,\xi_t)^{-1}f(\s_M(\xi_t))\pi(\s_M,\xi_t))
\end{equation}
is smooth as a function of $t\in U_T$, for any choice of $f$.
(Recall that $\pi(\mathfrak{s}_M,\xi)$ is smooth and
invertible as a function of $\xi\in\Xi_u$,
cf. Theorem \ref{thm:indint})
But if we choose $f$ such that $f(\xi)=\operatorname{Id}_V$ we see that
this is impossible. This proves (iii).

Let us finally prove (iv). We use the same set-up as above in the
proof of (iii), but now with $\chi_\a^{-1}f=ch$ for some smooth
section $h$. The equation (\ref{eq:ord2}) now becomes
\begin{equation}\label{eq:ord1}
\chi_\a^{-1}(\xi_t)(f(\xi_t)-
\a(t)\pi(\s_M,\xi_t)^{-1}f(\s_M(\xi_t))\pi(\s_M,\xi_t)),
\end{equation}
and again we know that this should be smooth as a function of
$t\in U_T$. This implies at $t=e$ that
\begin{equation}\label{eq:comm}
f(\xi)-\pi(\s_M,\xi)^{-1}f(\xi)\pi(\s_M,\xi)=0
\end{equation}
for all smooth sections $f$ supported on $U$. But for any
$A\in\operatorname{End}(V)$ there exists such a smooth section
with $f(\xi)=A$, thus equation (\ref{eq:comm}) implies that
$\pi(\s_M,\xi)$ is a scalar.
\end{proof}
\begin{dfn}\label{dfn:Wnul} Let $\xi\in\Xi_u$.
We denote by $\W^m_{\xi,\xi}\subset \W_{\xi,\xi}$ the subgroup generated by
the mirror reflections $\s_M$ with $M\in\M$ such that $\xi\in\M$.
The subgroupoid $\W^m$ whose set of objects is $\Xi_u$ and whose
set of arrows consists of the union of the sets $\W^m_{\xi,\xi}$ is
a normal subgroupoid of $\W$.
\end{dfn}
The statement that $\W^m$ is normal in $\W$ (i.e. invariant for
conjugation in $\W$) follows immediately from the fact that
$\M$ is $\W$-invariant.

The isotropy group $\W_{\xi,\xi}$ acts linearly on $\mathfrak{a}^P$ by
identifying $\mathfrak{a}^P$ with the tangent space of $\Xi_{(P,\d),u}$
at $\xi$ via the local diffeomorphism
\begin{align}
\mathfrak{a}^P &\to \Xi_{(P,\d),u}\\
x &\to(\xi)_{\exp(2\pi i x)}
\end{align}
centered at $\xi$.
\begin{dfn} Let $\xi\in\Xi_{(P,\d),u}$ and
consider the subset $R^{(P,\d)}_\xi\subset R^P$ consisting of the roots
$(P,\a)$ such that the zero set of the function $\chi_\a$ defined
by $\chi_\a(\xi_t)=\a(t)-1$ is locally near $\xi$ equal to a mirror
$M_{(P,\a),\xi}$ containing $\xi$.
\end{dfn}
\begin{prop}\label{prop:root}
The set $R^{(P,\d)}_\xi\subset\mathfrak{a}^{P,*}$
is a reduced integral root system such that
$\W^m_{\xi,\xi}\simeq W(R^{(P,\d)}_\xi)$.
\end{prop}
\begin{proof}
Recall the definitions of Appendix \ref{app:weylgroupoid}.
The group  $\W^m_{\xi,\xi}$ is by definition generated by the mirror
reflections in the mirrors of the form $M_{(P,\a),\xi}$ with
$(P,\a)\in R^{(P,\d)}_\xi$, and it is clear that $R^{(P,\d)}_\xi$
is invariant for $\W^m_{\xi,\xi}$. Therefore up to normalization
we see that $R^{(P,\d)}_\xi$ is the root system of the finite real
reflection group $\W^m_{\xi,\xi}$ (in the sense of \cite[Section 2.2]{C}).
Let us now consider the integrality of this root system.
By Theorem \ref{thm:cbs}, for all $(P,\tilde\a)\in R^{(P,\d)}_\xi$
the $\Wf_P$-orbit of $M_{(P,\tilde\a),\xi}$ contains a mirror
$M_{(Q,\ga),\xi^\prime}$ such that $(Q,H_\ga):=\operatorname{Ker}(Q,\ga)$
is the hyperplane in $\af^Q$ associated to a simple root
$\ga\in F_0\backslash Q$.
This implies that the $\Wf_P$-orbit of $\tilde\a$ contains a root
$\tilde{\ga}=w\tilde\a\in R_{Q\cup\{\ga\}}$ such that
$\tilde\ga\in\ga+\mathbb{Z}R_Q$. If we put $\a=w^{-1}\ga$ then we have
$(P,H_{\tilde\a})=(P,H_\a)$, and moreover $P\cup\{\a\}$ is a system of
simple roots for a (possibly non-standard) parabolic subsystem of
roots. Moreover, from Theorem
\ref{thm:mirror} we see that the elementary conjugation $\s^P_{P^\prime}$
of this parabolic root system leaves $P$ fixed, i.e. $P$ is self opposed in
$P^\prime:=P\cup\{\a\}$ (see paragraph \ref{subsub:eltconj}).
Now let us fix a $\W^m_{\xi,\xi}$-invariant
inner product in $\mathfrak{a}^P$. We need to show that
$\langle\overline{\a},\overline{\b}\rangle\in\mathbb{Z}$ holds
for all $\overline{\a}:=(P,\a)\in R^{(P,\d)}_\xi$ and
$\overline{\b}:=(P,\b)\in R^{(P,\d)}_\xi$. But since both $\a$ and
$\b$ can be replaced by roots which form a simple system
of roots together with $P$ such that $P$ is self opposed in
these systems, the integrality assertion follows from the
proof of Theorem 10.4.2 of \cite{C}: Let $M=M_{(P,\a),\xi,u}$ then
\begin{equation}
\s_{M}(\b):=\s^{P\cup\{\a\}}_{P}(\b)=w_{P\cup\{\a\}}w_P(\b)
\in w_{P\cup\{\a\}}(\b+\af_P)=(\b+\l\a)+\af_P
\end{equation}
with $\l\in\mathbb{Z}$ as desired.
Recall that $R^P$ consists only of {\it primitive} restrictions of
roots of $R_0\backslash R_P$. Therefore it is now clear that
$R^{(P,\d)}_\xi$ is integral and reduced.
\end{proof}
\begin{dfn}
Let $R^{(P,\d)}_{\xi,+}=R^{(P,\d)}_\xi\cap R^P_+$, and let
$\mathfrak{a}^{P,+}_\xi\subset \mathfrak{a}^P$ be the positive
Weyl chamber of $R^{(P,\d)}_{\xi,+}$.
We define
\begin{equation}
\mathfrak{R}_\xi=\{w\in\W_{\xi,\xi}\mid
w(\mathfrak{a}^{P,+}_\xi)=\mathfrak{a}^{P,+}_\xi\}.
\end{equation}
\end{dfn}
\begin{prop}
The subgroup $\mathfrak{R}_\xi\subset \W_{\xi,\xi}$ is a
complement for the normal subgroup $\W^m_{\xi,\xi}$. Hence
\begin{equation}
\W_{\xi,\xi}=\mathfrak{R}_\xi\ltimes\W^m_{\xi,\xi}.
\end{equation}
\end{prop}
\begin{proof}
The group $\W_{\xi,\xi}$ preserves the set $R^{(P,\d)}_{\xi}$ of roots
of the finite reflection group $\W^m_{\xi,\xi}$, and thus the choice
of a positive Weyl chamber induces a splitting of $\W_{\xi,\xi}$
as indicated.
\end{proof}
\section{The Knapp-Stein linear independence theorem}\label{sec:main}
In this section we will prove the Knapp-Stein linear independence
in the present context of affine Hecke algebras \cite{KS}, \cite{S}).

Let $\eta=(P,\d,t)\in\Xi_{(P,\d),u}$ be an $R_P$-generic
(cf. \cite[Definition 2.5]{DO}) induction parameter
in a small open $\W_{\xi,\xi}$-invariant neighborhood $U$ of
$\xi\in\Xi_{(P,\d),u}$. Let $W_Pr$ denote the central character of
$\d$. We will need to use Lusztig's first reduction theorem, in the
version as discussed in \cite{DO}; we refer the reader to
\cite[Section 2.6]{DO} and \cite{Lu} for further details.
The reduction theorem describes the
structure of the formal completion of $\H$ at the central character
$W_0(rt)$, as a matrix algebra with coefficients in the formal completion
at $\omega_t=W_P(rt)=tW_Pr$ of the Levi subalgebra $\H^P$.
The orbit $W_0(rt)$ is partitioned in equivalence classes of the form
$w\omega_t$ with $w\in W^P$. For each equivalence class
$w\omega_t$ there exits an idempotent $e_{w\omega_t}$ in the formal completion
$\bar{\H}_{W_0(rt)}$ of $\H$ at the central character $W_0(rt)$.
These idempotents form a complete orthogonal set of idempotents in
$\bar{\H}_{W_0(rt)}$. In the present context of $R_P$-generic
induction parameters the reduction theorem asserts that
$e_{\omega_t}\bar{\H}_{W_0(rt)} e_{\omega_t}=
e_{\omega_t}\bar\H^P_{\omega_t}$ where $e_{\omega_t}$
is a central idempotent on the right hand side, and that we
have a decomposition
\begin{equation}\label{eq:dec}
\bar\H_{W_0(rt)}=\bigoplus_{u,v\in W^P}
\iota_u^0{e_{\omega_t}}\bar\H^P_{\omega_t}\iota_{v^{-1}}^0
\end{equation}
which yields an isomorphism of $\bar\H_{W_0(rt)}$ and a
matrix algebra of size $N=|W^P|$ and coefficients in
$e_{\omega_t}\bar{\H}_{\omega_t}^P$.
The theorem moreover asserts in this situation
that if $w(P)=Q\in\P$ then the conjugation map
$c_w:x\to \iota_w^0 x\iota_{w^{-1}}^0$ is well defined on
$e_{\omega_t}\bar\H^P_{\omega_t}\subset \bar\H_{W_0(rt)}$
and defines an algebra isomorphism
\begin{equation}
c_{\iota_w^0}:e_{\omega_t}\bar\H^P_{\omega_t}\isom e_{w(\omega_t)}\bar\H^Q_{w(\omega_t)}
\end{equation}
which coincides with the isomorphism originating from the isomorphism of
root data $\Ri^P\isom\Ri^Q$ induced by $w$.

We also use the concept of the \emph{constant term}
of (matrix coefficients of) tempered representations along a standard
parabolic subset $P\in\P$, see \cite[Section 3.6]{DO} and \cite[Section 6]{DO}.
The subset of $w\in W^P$ such that
$e_{w\omega_t}$ contributes to the constant term $V_{\eta}^P$
along $P$ of the tempered module $V_\eta$ is equal to $W_{P,P}$
(cf. \cite[Proposition 6.12]{DO}, where we remark that $W_{P,P}=D^{P,P}$ in the
situation $Q=P$). By the Morita equivalence Proposition \ref{prop:morita}
we see that $\W_{\xi,\xi}\simeq\Wf_{\K\xi,\K\xi}\subset W_{P,P}$.
We will identify $\W_{\xi,\xi}$ with this subgroup of $W_{P,P}$ in
the rest of this section.

Choose a complete set $S$ of
representatives for the left cosets of $\W_{\xi,\xi}$
in $\W_{P,P}$. For each $s\in S$ and $\r\in\mathfrak{R}_\xi$ we define
\begin{equation}\label{eq:idem}
E_{s,\r;t}=\sum_{w\in\W^m_{\xi,\xi}}e_{sw\r\omega_t}\in\bar\H_{W_0(rt)},
\end{equation}
which is an idempotent of the formal completion of $\H$ at the
central character $W_0(rt)$. Recall that $rt$ is $R_P$-generic.
\begin{prop}\label{prop:regproj}
For all $s\in S$, $\r\in\mathfrak{R}_\xi$ and $\eta\in U$ we define a
projection $p(s,\r,\eta)$
in $V_\eta=i(V_\d)$ by $p(s,\r,\eta):=\pi(\eta,E_{s,\r;t})$.
\begin{enumerate}
\item Viewed as rational function of $\eta$, $p(s,\r,\eta)$ is
regular for $\eta\in U$.
\item For all $\eta\in U$, $\sum_{s\in S,\r\in\mathfrak{R}_\xi}p(s,\r,\eta)$
is the projection onto $V_{\eta}^P$, the constant part of
$V_{\eta}$ along $P$ (see \cite[Section 3.6]{DO} for the definition of
the constant part of a tempered module).
\item The collection of idempotents $\{p(s,\r,\eta)\}$ is mutually orthogonal
in $\operatorname{End}(i(V_\d))$ for all $\eta\in U$.
\item For all $s,\r$, and all $\eta\in U$: $p(s,\r,\eta)$ is
an endomorphism of the $\H^P$-module structure on $i(V_\d)$ obtained by
restricting $\pi(\eta)$ to $\H^P$.
\end{enumerate}
\end{prop}
\begin{proof}
For generic $\eta$ the properties (ii), (iii) and (iv) follow straightforward from
the definitions and from \cite{DO}, Sections 3.6, 6.1, 6.2 (especially Corollary 6.9,
in which one should observe that $\W_{P,P}=\{d\in D^{P,P}\mid d(P)=P\})$ and 6.3.
From this remark it is clear that (i) implies (iii), (iv). But \cite[Proposition 7.8]{DO}
implies that the projection $V_{\eta}\to V_{\eta}^P$ is also smooth
in $\eta\in\Xi_{(P,\d),u}$, and thus also (ii) will follow from the generic
case provided we know (i). Thus it remains only to prove (i).

It is obviously enough to consider the case $s=e$ and $\mathfrak{r}=e$
(replace $\xi$ by $s\xi$ and $t$ by $s\mathfrak{r}t$ in (\ref{eq:idem})).
We compute
a matrix coefficient of $P(\eta):=p(e,e,\eta)$. Let $a,b\in i(V_\d)$.
Then, using the notations of \cite[Subsection 6.2, 6.3 and 6.4 ]{DO} we have:
\begin{align}\label{eq:W0inv}
\langle a,P(\eta) b\rangle&=f_{a,P(\eta)b}(\eta,1)\\
\nonumber &=\sum_{d\in \W^m_{\xi,\xi}}f^d_{a,b}(\eta,1)\\
\nonumber &=\sum_{d\in \W^m_{\xi,\xi}}f^1_{\pi(d,\eta)a,\pi(d,\eta)b}(d\eta,1)\\
\nonumber &=\sum_{d\in \W^m_{\xi,\xi}}f^1_{a,b}(d\eta,1)\\
\nonumber &=\sum_{d\in
\W^m_{\xi,\xi}}c(d\eta)(c(d\eta)^{-1}f^1_{a,b}(d\eta,1)).
\end{align}
Here the first two equalities follow form a direct unwinding of definitions,
the third equality follows
from an application of \cite[Lemma 6.14]{DO}, the fourth is the unitarity
property \cite[Theorem 4.33]{O} of the intertwiners $\pi(d,\eta)$,
and the last one is trivial. By \cite[Theorem 6.18]{DO} the
expression $c(d\eta)^{-1}f^1_{a,b}(d\eta,1)$ is regular for
$\eta$ in a small tubular neighborhood of $\Xi_{(P,\d),u}$.
By Theorem 6.3 (iii), the singularities of $c(d\eta)$ for
$d\in\W^m_{\xi,\xi}$ and for $\eta\in U$ are poles of order at most
one along the mirrors $M$ of $\W^m_{\xi,\xi}$ which contain $\xi$.
On the other hand, the expression on the right hand
side of the last equality of equation (\ref{eq:W0inv})
also shows that $\langle a,P(\eta)
b\rangle$ is $\W^m_{\xi,\xi}$-invariant as a function of $\eta$.
The product of the function
$U\ni\eta\to\langle a,P(\eta) b\rangle$ by
\begin{equation}
\pi:=\prod_{(P,\a)\in R^{(P,\d)}_{\xi,+}}(P,\a)
\end{equation}
extends to a $\W^m_{\xi,\xi}$-skew invariant analytic function on $U$.
It is a well known basic fact from the invariant theory of finite real
reflection groups that a $\W^m_{\xi,\xi}$-skew invariant analytic function
on $U$ is divisible by $\pi$. This implies that the apparent first order poles of
$\langle a,P(\eta) b\rangle$ along the mirrors of $\W^m_{\xi,\xi}$
are removable themselves. Therefore $\langle a,P(\eta)
b\rangle$ extends to an analytic function of $\eta\in U$.
\end{proof}
\begin{cor}\label{cor:subVP}
\begin{enumerate}
\item
For all $\eta\in U$ we have a decomposition
\begin{equation}
V_{\eta}^P=\bigoplus_{s\in S,\r\in\mathfrak{R}_\xi}V^P_{s,\r,\eta}
\end{equation}
of the $\H^P$-module $V_\eta^P$ as a direct sum of
$\H^P$-submodules $V^P_{s,\r,\eta}$ defined by
$V^P_{s,\r,\eta}:=p(s,\r,\eta)(V_\eta)$.
\item
For all $s\in S,\r\in\mathfrak{R}_\xi$, all the irreducible subquotients of the
finite length $\H^P$-module $V^P_{s,\r,\xi}$ are isomorphic to
$(\d^s)_{st^0}$.
\end{enumerate}
\end{cor}
\begin{proof}
(i) This is a direct consequence of Proposition \ref{prop:regproj}.

(ii) Assume that $t$ is such that $\eta=(P,\d,t)\in U$ is generic.
We have according to \cite[equation (3.6)]{DO} that
\begin{equation}
i(V_\d)\simeq\bigoplus_{u\in W^P}\iota_u^0e_{\omega_t}
\otimes V_{\d_t}.
\end{equation}
From (\ref{eq:idem}) and the orthogonality of the idempotents we
then conclude that
\begin{equation}
V^P_{s,\r,\eta}=\bigoplus_{w\in \W^m_{\xi,\xi}} e_{s w
\mathfrak{r}\omega_t}\iota_{s w \mathfrak{r}}^0 \otimes V_{\d_t}
\end{equation}
Using the definition of the normalized intertwining operators
(see Remark \ref{eq:normint2} and \cite[equation (3.7)]{DO}) we see
that this is isomorphic to
\begin{equation}
V^P_{s,\r,\eta}=\bigoplus_{w\in \W^m_{\xi,\xi}}
\pi(\mathfrak{r}^{-1} w^{-1} s^{-1},sw\mathfrak{r}(\eta))
(e_{\omega_{sw\mathfrak{r}t}}\otimes V_{\d^s_{sw\mathfrak{r}t}}),
\end{equation}
so that the $\H^P$-module $V^P_{s,\r,\eta}$ is isomorphic to a direct sum
of the irreducible modules $\H^P$-modules $(\d^{s})_{sw\r t}$, where
$w$ runs over all the elements of $\W^m_{\xi,\xi}$.
If we substitute $t=t_0$ (this corresponds to
taking $\eta=\xi$) then each of these irreducible summands
coincides with $(\d^s)_{st^0}$. Hence the character of $V^P_{s,\r,\xi}$
(since $V^P_{s,\r,\eta}$ depends smoothly on $\eta$) is simply
$|\W^m_{\xi,\xi}|$ times the character of
$(\d^s)_{st^0}$. Therefore all the irreducible subquotients of the
finite length $\H^P$-module $V^P_{s,\r,\xi}$ are isomorphic to
$(\d^s)_{st^0}$.
\end{proof}
\begin{cor}\label{cor:uniqueirrsub}
The $\H^P$-module $V^P_{s,\r,\xi}$ has a unique irreducible
submodule, which is isomorphic to $(\d^s)_{st^0}$ for all
$\r\in\mathfrak{R}_\xi$ and all $s\in S$.
\end{cor}
\begin{proof}
By symmetry it suffices to prove this for $s=e$.

By Frobenius reciprocity \cite[Proposition 3.18]{DO} we have:
\begin{equation}\label{eq:frobrec}
\operatorname{End}_\H(\pi(\xi))=\operatorname{Hom}_{\H^P}(\d_{t_0},V_\xi^P).
\end{equation}

By \cite[Corollary 5.4]{DO} we know that $\operatorname{End}_\H(\pi(\xi))$
is the complex linear span of the operators $\pi(g,\xi)$ with
$g\in\W_{\xi,\xi}$. We have seen that $\pi(g,\xi)$ is a scalar for
$g\in\W^m_{\xi,\xi}$. Thus $\operatorname{End}_\H(\pi(\xi))$ is
already spanned by the operators $\pi(\r,\xi)$ with
$\r\in\mathfrak{R}_\xi$.
Hence the dimension of $\operatorname{End}_\H(\pi(\xi))$
is at most $|\mathfrak{R}_\xi|$.

On the other hand, Corollary \ref{cor:subVP} implies that
$(\d)_{t^0}$ occurs at least once as a submodule of
$V^P_{e,\r,\xi}$, for every $\r\in\mathfrak{R}_\xi$.
Combining this with the Frobenius reciprocity formula (\ref{eq:frobrec})
we obtain that $(\d)_{t^0}$ occurs precisely once as a submodule of
$V^P_{s,\r,\xi}$ for every $\r\in\mathfrak{R}_\xi$.
Again invoking Corollary \ref{cor:subVP} we
conclude that this irreducible submodule is in fact the unique
irreducible submodule of $V^P_{s,\r,\xi}$.
\end{proof}
\begin{thm}\label{thm:mainR}
For all $\xi\in\Xi_u$, we have
\begin{equation}
\operatorname{End}_\H(\pi(\xi))=\sum_{\r\in\mathfrak{R}_\xi}\C\pi(\r,\xi),
\end{equation}
the complex linear span of the operators $\pi(r,\xi)$ with
$r\in\mathfrak{R}_\xi$. Moreover, these operators are linearly
independent, so that
\begin{equation}
\operatorname{dim}\operatorname{End}_\H(\pi(\xi))=|\mathfrak{R}_\xi|.
\end{equation}
\end{thm}
\begin{proof}
In the course of the proof of Corollary \ref{cor:uniqueirrsub}
it was shown that the dimension of $\operatorname{End}_\H(\pi(\xi))$
is in fact precisely equal to $|\mathfrak{R}_\xi|$. It was
also remarked in the proof of Corollary \ref{cor:uniqueirrsub} that the
space of endomorphisms of $\pi(\xi)$ is spanned by the operators
$\pi(\r,\xi)$ with $\r\in\mathfrak{R}_\xi$.
\end{proof}
\begin{thm}\label{thm:Rmorequiv}
Let $\xi\in\Xi_u$ be a standard induction datum for $\H$.
\begin{enumerate}
\item Let $\ga_\xi$ denote the restriction of the $2$-cocycle
$\ga_{\W,\Xi}$ to $\W_{\xi,\xi}$. This $2$-cocycle is cohomologous to
the pull-back of a $2$-cocycle on $\mathfrak{R}_\xi$ (which we also
denote by $\ga_\xi$ by abuse of notation)
via the natural projection $\W_{\xi,\xi}\to\mathfrak{R}_\xi$.
\item The map
$\mathfrak{R}_\xi\ni\mathfrak{r}\to\pi(\mathfrak{r},\xi)$
extends by linearity to an algebra isormorphism $\pi(\cdot,\xi)$
from the $\ga_\xi$-twisted complex group algebra
${}^\ga\mathbb{C}[\mathfrak{R}_\xi]$ of $\mathfrak{R}_\xi$ to the
commutant algebra $\operatorname{End}_\H(\pi(\xi))$.
\item Up to the choice of the isomorphism $\pi(\cdot,\xi)$
(which depends on our choice of normalized intertwining operators
(\ref{eq:normintdefn})) there exists a unique bijection
\begin{align}
\widehat{{}^\ga\mathbb{C}[\mathfrak{R}_\xi]}&\to\hat\S_{\W\xi}\\
\nonumber \rho&\to\pi_\rho
\end{align}
(where $\hat\S_{\W\xi}$ denotes a complete set of representatives of
the finite set of isomorphism classes of irreducible
representations of $\S$ with central character $\W\xi$ (central character
in the sense of \cite[Corollary 5.5]{DO}))
such that we have a decomposition (recall that $V_\xi$ is the
vector space on which $\pi(\xi)$ is realized)
\begin{equation}
V_{\xi}=\bigoplus_{\rho}\pi_\rho\otimes\rho
\end{equation}
as a $\S-{}^\ga\mathbb{C}[\mathfrak{R}_\xi]^{op}$-bimodule. Here
the sum runs over $\rho\in\widehat{{}^\ga\mathbb{C}[\mathfrak{R}_\xi]}$
(viewed as irreducible right
${}^\ga\mathbb{C}[\mathfrak{R}_\xi]^{op}$-module).
\end{enumerate}
\end{thm}
\begin{proof}
The first claim (i) comes from the fact that the projective
representation $\mathbb{P}\pi(\xi)$ on $\mathbb{P}(V_\xi)$
descends to $\mathfrak{R}_\xi$ by Theorem \ref{thm:mirror}(iv).
Hence $\pi(\mathfrak{r},\xi)\in\operatorname{U}(V_\xi)$
is a lifting of
$\mathbb{P}\pi(w^0\mathfrak{r},\xi)\in\mathbb{P}\operatorname{U}(V_\xi)$
for all $w^0\in\W_{\xi,\xi}^0$. Therefore $[\ga_\xi]$ descends to
$\mathfrak{R}_\xi\times\mathfrak{R}_\xi$.
The remaining part of the Theorem follows
from Theorem \ref{thm:mainR} by the arguments as in \cite[p. 87-88]{A}.
\end{proof}
\section{The cocycle $\ga$ for classical Hecke algebras}\label{sec:coc}
In this section we prove the triviality of the $2$-cocycles $\ga_{\W,\Xi_u}$
for the classical Hecke algebras. The computation is based on the classification
\cite{OpSo} of the discrete series representations of classical affine
Hecke algebras.
\begin{thm}\label{thm:general}
Let $\Ri=(X,Y,R_0,R_0^\vee,F_0)$ be an irreducible root datum of classical type,
and let $q$ be an arbitrary positive parameter function for $\mathcal{R}$.
\begin{enumerate}
\item[(i)] We have $[\ga_\xi]=1$ for all $\xi\in\Xi_u$.
\item[(ii)] We have $[\ga_{\W,\Delta}]=1$ if $R_0$ is not of type
$\textup{D}_n$ with $n > 8$, or if $X$ is not the root lattice of $R_0$.
\end{enumerate}
\end{thm}
\begin{proof}
We use that the assertion $[\ga_{\W,\Delta}]=1$ of (ii) implies (i) by taking
the pull-back to $\W_{\Xi_u}$ (using Definition \ref{dfn:gaxi}).
Hence in all cases mentioned in (ii) it suffices to prove (ii).
This is an aggregate of various special cases which are treated separately below.
In the remaining case $R_0=\textup{D}_n$ with $n>8$ and $X$ the root lattice of
$R_0$ we will show directly that (i) holds.
\end{proof}
\begin{rem}
Remarkably, if $R_0$ is of type $\textup{D}_n$
with $n>8$ and $X$ equals the root lattice of $R_0$ then
there exist discrete series representations $\delta_P$ of $\H_P$
such that $[\ga_{(P,\delta_P)}]\not=1$.
\end{rem}
We will not give the proof here of the nontriviality of $[\ga_{(P,\delta_P)}]$
in this exceptional case, but for a precise formulation see paragraph
\ref{subsub:nontriv}.
Using the description of the discrete series for affine
Hecke algebras of type $\textup{D}_n$ as given in \cite{OpSo},
the proof consists of tracing the action of $\W_{\xi,\xi}$
on the isomorphisms constructed in Lusztig's first reduction
theorem \cite{Lu}.
\subsection{A remark on isogenous affine Hecke algebras}
For later use we list the following useful general fact.
\begin{lem}\label{lem:ind} Let $\H$ be a semisimple affine
Hecke algebra. Let $\mathcal{R}=(X,Y,R_0,R_0^\vee,F_0)$ be the
based root datum of $\H$, and $q$ its parameter function. Let
$\H^e$ be an isogenous extension of $\H$, i.e. a semisimple affine Hecke algebra
with root datum $\mathcal{R}^e=(X^e,Y^e,R_0,R_0^\vee,F_0)$ where $X\subset X^e$
is an extension of lattices such that $R_{nr}=R_{nr}^e$, and such that the parameter
$q^e$ of $\mathcal{H}^e$ is equal to $q$, viewed as functions on $R_{nr}^\vee$.
Then $\H\subset\H^e$ is an isometric embedding of $\H$ as a $*$-subalgebra of
$\H^e$ of finite index, and the induction and restriction functors between
$\textup{Rep}(\H)$ and $\textup{Rep}(\H^e)$ send irreducible unitary (resp.
discrete series) representations to finite direct sums of irreducible unitary
(resp. discrete series) representations.
\end{lem}
\begin{proof} Let $\Lambda_r\subset X$ be the
root lattice of $R_0$. We introduce the finite abelian groups
$\Omega=X/\Lambda_r$ and $\Omega^e=X^e/\Lambda_r$,
and we denote by $\H_r\subset \H^e$ the affine Hecke algebra
with root system $R_0$ (and basis $F_0$) and whose root datum
has the root lattice $\Lambda_r$ of $R_0$ as lattice $X$.
As is well known, we have $\H^e=\H_r\rtimes\Omega^e$ and
$\H=\H_r\rtimes{\Omega}\subset\H^e$. From this and the standard
description of the Hilbert algebra structures ($*$ and the trace $\tau$)
on $\H$ and $\H^e$ we see that $\H\subset \H^e$ is an isometric embedding
of a $*$-subalgebra. It is immediate from this that the restriction
functor preserves unitarity and discreteness.

Now let us consider induction. Multiplication
yields an action of $\Omega^e$ on $\H^e$ which permutes the standard
orthonormal basis elements $\{N_w\}_{w\in W_P^e}$ of $\H^e$ freely
(with $W^e$ the extended affine Weyl group with root system $R_0$).
Choose a set $\omega_1,\dots,\omega_n$ of representatives for the
cosets of $\Omega$ in $\Omega^e$.
Then we have an orthogonal decomposition
$\H^e=\omega_1\H\oplus\dots\oplus\omega_n\H$. For all $h\in\H$
we have $\omega_ih=h^{\omega_i}\omega_i$, where $h\to h^{\omega_i}$
is a (special) affine diagram automorphism associated with $\omega_i$
of $\H$. Such automorphisms are isometries since they permute the
standard orthonormal basis of $\H$.
Let $(V,\pi)$ be a finite dimensional representation of $\H$.
The underlying vector space $i(V)=\H^e\otimes V$ of the
induced representation $i(\pi)$ is the direct sum of the subspaces
$V_i:=\omega_i^{-1}\otimes V$, and we identify each $V_i$ with $V$ by
the map $V\ni v\mapsto \omega_i^{-1}\otimes v\in V_i$. If $(V,\pi)$ is
unitary then $(i(V),i(\pi))$ unitary with respect to the Hilbert
space structure on $i(V)$ which is the orthogonal direct sum of
the subspaces $V_i$ (each equipped with the transfer of the
Hilbert structure of $V$ under this identification). We see that
the character $\chi_{i(\pi)}$ satisfies
$\chi_{i(\pi)}(\omega_i h)=\chi_\pi(h^{\omega_i})$ for each $i$.
With the above orthogonal decomposition of $\H^e$ and the fact that
$h\to h^{\omega_i}$ is an isometry of $\H$ we see that $i(\pi)$ is
a discrete series character if and only if $\pi$ is a discrete series
character.
\end{proof}
\subsection{$R_0$ has only irreducible components of type $\textup{A}$}
In this situation we prove a more general result:
\begin{prop}\label{prop:A} Let $R_0$ be a root system whose irreducible components
are all of type $\textup{A}$, and let $\mathcal{R}=(X,Y,R_0,R_0^\vee,F_0)$ be
an arbitrary (not necessarily semisimple) root datum whose underlying root system
is $R_0$. Then $\ga:=\ga_{\W,\Delta}=1$.
\end{prop}
\begin{proof}
If $R_0$ has only irreducible components of type $\textup{A}$ then the same
is true for any of its standard parabolic subsystems $R_P\subset R_0$.
Let $\H_P^e:=\H_P(X^e,q^e)$ be the extended semisimple affine Hecke algebra
whose root datum
has underlying root system $R_P$ with basis $F_P\subset F_0$ and whose lattice
$X_P^e=\Lambda_P$ is the weight lattice of $R_P$.
Then $\H_P^e$ is a tensor product of various
extended affine Hecke algebras
$\H^{A,e}_{\lambda_i-1}:=\H^{A}_{\lambda_i-1}(\Lambda^A_{\lambda_i-1},q)$ of type
$\textup{A}_{\lambda_i-1}$ (for various $\lambda_i\geq 2$).
Since it is well known \cite{Z} that the irreducible discrete series
representations of
$\H^{A,e}_{\lambda_i-1}$ all have dimension $1$, it follows from the above
that all the discrete series representations $\delta_P^e$ of $\H_P^e$ are
one dimensional. In particular, the restriction of an irreducible
discrete series representation of $\H_P^e$ to $\H_P$ is irreducible.
By Lemma \ref{lem:ind} and by the adjointness of restriction and induction
we see that all irreducible discrete series representations of $\H_P$ are
obtained in this way, and in particular are one dimensional. The
projective representation $\tilde{\delta}$ of $\W_{\Delta}$ is thus one dimensional,
hence trivial. In particular, its class $[\ga_{\W,\Delta}]$ is trivial.
\end{proof}
\subsection{$R_0$ of type $\textup{B}_n$}
The next result generalizes a result of Slooten \cite{Sl1}.
\begin{prop}\label{prop:BQ}
Suppose that $\mathcal{R}$ is of type $\textup{C}_n^{(1)}$ (i.e. $R_0$ is of type
$\textup{B}_n$ and $X$ is the root lattice of $R_0$) with arbitrary
positive parameter function $q=(q_0,q_1,q_2)$,
where (using the standard realization for $C_n^{(1)}$)
$q_0=q(s_{2x_n})$, $q_1=q(s_{x_i-x_{i+1}})$ and $q_2=q(s_{1-2x_1})$.
Then $[\ga_{W,\Delta}]=1$.
\end{prop}
\begin{proof}
In order to analyze the
cocycle $\ga_{\W,\Delta}$, let us first look more carefully at the
type $\textup{A}$ case. Let $\H^A_{n-1}(X,q)$ be an affine Hecke algebra
with $R_0$ of type $\textup{A}_{n-1}$ and with lattice $X$ (situated between
the root lattice $\Lambda_r$ and the weight lattice $\Lambda$ of $R_0$)
and parameter $q\not=1$ (if $q=1$ there are no
discrete series). If $q\not=1$ then the set $\Delta^A_{n-1}(X,q)$
of equivalence classes of discrete series representations of the affine
Hecke algebra of type $\textup{A}_{n-1}$ with lattice $X$ is in canonical
bijection with the set $K_{n-1}^A(X)$ of characters of $X$ which are
trivial on the root lattice $\Lambda_r$ of $R_0$, through the central character map.
Namely, in terms of the notation of \cite[Section 8]{OpSo} we have
$K_{n-1}^A(X)=\Gamma^{max}/\Gamma$ where
$\Gamma^{max}=\textup{Hom}(\Lambda/\Lambda_r,\mathbb{C}^\times)\approx C_n$.
The group $\Gamma^{max}$ acts simply transitively on
the set of vertices $E(C^\vee)$ (notations as in loc. cit.). Hence the group
$K$ acts simply transitively on the set of $\Gamma$-orbits on $E(C^\vee)$, and
for each $s\in E(C^\vee)$ we have $\Gamma_s=1$. By \cite[Theorem 8.7]{OpSo} the
set $\Delta^A_{n-1}$ is a disjoint union over all $\Gamma$-orbits of $E(C^\vee)$
of the set of discrete series characters of the graded affine Hecke algebra
$\mathbf{H}(R_{s(e),1},V,F_{s(e),1},k_e)$. In the type $\textup{A}_{n-1}$-case, these
are all isomorphic to a graded affine Hecke algebra of type $\textup{A}_{n-1}$, and
hence each of these contributes precisely one discrete series character
(since we assume that $q\not=1$, implying that $k_e=k\not=0$).
If $k\in K^A_{n-1}(X)$ we denote by $\delta_k$ the unique,
one-dimensional discrete series character of the type $\textup{A}_{n-1}$ affine Hecke algebra
$\H^A_{n-1}(X,q)$  whose central character has unitary part $k$. Then
$\Delta^A_{n-1}(X)=\{\delta_k\mid k\in K^A_{n-1}(X)\}$.
Through the twisting automorphisms $\psi_k$ (see paragraph \ref{subsub:groupoid})
the group
$K_{n-1}^A(X)$ acts on $\Delta_{n-1}^A(X)$. We have
\begin{equation}
\delta^k_{k^\prime}:=\delta_{k^\prime}\circ\psi_{k}^{-1}=\delta_{k^\prime k^{-1}}
\end{equation}
Now we return to
the case where $\mathcal{R}$ is equal to $C^{(1)}_n$.
The possible pairs $(P,\delta_P)$ with $\delta_P\in\Delta_P$ can be described
explicitly as follows.
First notice that $P\subset F_0$ is of type
\begin{equation}\label{eq:Atype}
\textup{A}_{\l_1-1}\times \textup{A}_{\l_2-1}\times\dots\times \textup{A}_{\l_r-1}\times
\textup{B}_l
\end{equation}
where $l\leq n$ and where
$\l=(\l_1, \l_2, \dots, \l_r)$ is a composition of $n-l$.
In this situation $\H_P$ is of the form
\begin{equation}\label{eq:tensor}
\H^A_{\l_1-1}(\Lambda_{\l_1-1}^A,q_1)\otimes \H^A_{\l_2-1}(\Lambda_{\l_2-1}^A,q_1)\otimes\dots\otimes
\H^A_{\l_r-1}(\Lambda_{\l_r-1}^A,q_1)\otimes \H^B_l(q)
\end{equation}
where $\H^A_{\l_i-1}(\Lambda_{\l_i-1}^A,q_1)$ denotes the extended affine Hecke algebra
of type $\textup{A}_{\l_i-1}$ whose lattice equals the weight lattice
$\Lambda_{\l_i-1}^A$ of $\textup{A}_{\l_i-1}$ and
with parameter $q_1$, and where $H^B_l(q)$ is the affine Hecke algebra of
type $\textup{C}_l^{(1)}$ with parameter $q$. Thus
\begin{equation}\label{eq:deltaP}
\delta_P=\delta_{1,k_1}\otimes\delta_{2,k_2}\otimes\dots\otimes\delta_{r,k_r}\otimes\delta
\end{equation}
where $k_i\in K^A_{\l_i-1}(\Lambda_{\l_i-1}^A)$ (a cyclic group of order $\l_i$),
$\d_i$ is the unique irreducible discrete series representation of
$\H^A_{\l_i-1}(\Lambda_{\l_i-1}^A,q_1)$ with real infinitesimal character, and where
$\delta$ is a discrete series representation of $\H^B_l(q)$.
As discussed in paragraph \ref{subsub:groupoid}, an element
$g=k\times u\in\W_{P,Q}=K_Q\times \Wf_{P,Q}$ gives rise
to an automorphism $\psi_g:\H_P\to\H_Q$.
In the present situation, $u$ is a composition
of a permutation of tensor legs of the form $\H^{A}_{\l_i-1}(\Lambda_{\l_i-1}^A,q_1)$
in (\ref{eq:tensor}) with equal $\lambda_i$, with a tensor product of
automorphisms of the tensor factors
$\H^{A}_{\l_i-1}(\Lambda_{\l_i-1}^A,q_1)$ induced by a diagram automorphisms
of the finite Dynkin diagram of type $\textup{A}_{\l_i-1}$.
Recall that
$K_Q=T^Q\cap T_Q\approx\textup{Hom}(X_Q/(X\cap\mathbb{R}Q),\mathbb{C}^\times)$
(by (\ref{eq:KP})).
Therefore $K_Q$ is the direct product of the cyclic
groups $K^{A}_{\l_i-1}(\Lambda_{\l_i-1}^A)$ of order $\l_i$, and
$k=k_i\times k_2\times\dots\times k_r$ acts by the tensor product of the
automorphisms $\psi_{k_i}$ described above. The crucial observation to make
here is that the action of $\psi_g$ on the last tensor leg is always trivial.
If we choose a basis vector for each of the one dimensional representations
$\d_{i,k_i}$ we obtain a natural identification of the vector space of $\delta_P$
with the vector space $V_\d$ on which $\d$ is realized.
With this identification, we can choose
$\psi_g=\operatorname{Id}_{V_\d}$ for all $g$, and hence
$\ga_{\W,\Delta}=1$.
\end{proof}
\begin{prop}\label{prop:BP}
Consider the irreducible \emph{extended} affine Hecke algebra
$\H:=\H_n^{B,e}$ with $R_0$ of type $\textup{B}_n$ and $X=\Lambda_n^B$,
the weight lattice of type $\textup{B}_n$. Then $[\ga_{\W,\Delta}]=1$.
\end{prop}
\begin{proof}
The proof of Proposition \ref{prop:BQ} changes only slightly. In the present
situation $q$ is restricted to the cases where $q_0=q_2$.
The last tensor leg of the algebra $\H_P$ in (\ref{eq:tensor}) changes to
the extended algebra $\H^{B,e}_l$. For future reference we note
that $\H_P$ is again of the form $\H_P=\H_P^{A,e}\otimes\H_P^{B,e}$, a tensor product
of a number of extended type $\textup{A}$ Hecke algebras with an extended type
$\textup{B}$ Hecke algebra (recall that the lattice $X_P=\Lambda_P$ underlying $\H_P$ is
the projection of the weight lattice $\Lambda$ onto the vector space $\mathbb{R}P$
which is indeed the weight lattice of $R_P$).
Accordingly we have $\delta_P=\delta_P^{A,e}\otimes\delta_P^{B,e}$
(analogous to (\ref{eq:deltaP})).

The group $K_P^{B,\Lambda}=T_P\cap T^P$ that needs to be considered in the
definition of $\W_\Delta$ equals, as before, the group of characters of $\Lambda_P$
which are trivial on the sublattice $\Lambda\cap \mathbb{R}P$.
In the present situation one checks
easily that if there exists at least one odd $\lambda_i$ then
$\Lambda\cap\mathbb{R}P=\Lambda_{P,r}$, the root lattice of $R_P$.
Therefore $K_P^{B,\Lambda}=K_P\times C_2$ in this case, where the group $K_P$
is defined as above for the previous case $\H_n^B:=\H(\textup{C}_n^{(1)},q)$
(i.e. a direct product of cyclic groups) and where the extra factor $C_2$
(the group with $2$ elements) acts (by twisting automorphisms) on rightmost
tensor factor $\H_P^{B,e}$ only.

On the other hand, if all $\lambda_i$ are even, then
$\Lambda\cap\mathbb{R}P=\Lambda_{P,r}+(v+\Lambda_{P,r})$ where $v\in\Lambda_P$ is
a vector having coordinates $\pm 1/2$ such that for each part $\lambda_i$ of $\lambda$
the corresponding coordinates of $v$ sum up to zero (which is possible because $\lambda_i$
is even). Hence in this case the group $K_P^{B,\Lambda}$ is equal to
the kernel of the unique quadratic character $\rho$ of $K_P\times C_2$ which is
nontrivial on all factors (recall that all factors are even cyclic groups,
hence admit a unique nontrivial quadratic character). Thus
$K_P^{B,\Lambda}$ is a subgroup of index $2$ in $K_P\times C_2$ in this
case.

It is at this point useful to remark that $\W_{\Delta}$ is equivalent to
a finite union of the isotropy groups $\W_{(P,\delta_P),(P,\delta_P)}$ by
choosing a complete set of representatives for the $\W_\Delta$-orbits
of pairs $(P,\delta_P)$. Hence it is enough to show that the restriction
$\ga_{P,\delta_P}$ of $\ga$ to $\W_{(P,\delta_P),(P,\delta_P)}$ is trivial
for each pair $(P,\delta_P)$.
Recall that $\delta_P=\delta_P^{A,e}\otimes\delta_P^{B,e}$.
By the above we see that
$\W_{(P,\delta_P),(P,\delta_P)}\subset \W^e_{(P,\delta_P),(P,\delta_P)}:=
\W_{(P,\delta_P^{A,e}),(P,\delta_P^{A,e})}^{A,e}\times\W_{(P,\delta_P^{B,e}),(P,\delta_P^{B,e})}^{B,e}$
(a subgroup of index at most $2$, depending on $P$),
where the factor
$\W_{(P,\delta_P^{B,e}),(P,\delta_P^{B,e})}^{B,e}
\subset K_l^B(\Lambda^B_l)=C_2$
(recall that the finite Dynkin diagram of type $\textup{B}_l$
has no nontrivial diagram automorphisms) is either trivial
or isomorphic to $C_2$ (depending on $\delta_P^{B,e}$). The projective
representation $\tilde{\delta}_P$ of $\W_{(P,\delta_P),(P,\delta_P)}$
(whose class is $\ga_{P,\delta_P}$) is the restriction of a projective
representation $\tilde{\delta}_P^e$ of $\W^e_{(P,\delta_P),(P,\delta_P)}$
(which is defined as usual, by twisting $\delta_P=\delta_P^{A,e}\otimes\delta_P^{B,e}$
with the automorphisms of $\H_P$ coming from $\W^e_{(P,\delta_P),(P,\delta_P)}$).
Observe that $\tilde{\delta}_P^e$ is the tensor product of a projective
representation $\tilde{\delta}^{A,e}_P$ of $\W_{(P,\delta_P),(P,\delta_P)}^{A,e}$
and a projective representation $\tilde{\delta}^{B,e}_P$ of
$\W_{(P,\delta_P),(P,\delta_P)}^{B,e}$. The first tensor factor is linear
because it has dimension $1$ (and thus is trivial as a projective
representation), and the second tensor factor is linear since
$H^2(C_2,\mathbb{C}^\times)=1$ (actually $H^2(C_n,\mathbb{C}^\times)=1$
for any finite cyclic group $C_n$, see e.g. \cite[Exercise XX 16]{La}
(Warning: the even and odd cases have been mixed up in loc.cit.)). Hence
by restriction we see that $[\ga_{P,\delta_P}]=1$, which is
what we needed to show.
\end{proof}
\subsection{$R_0$ of type $\textup{C}_n$}
\begin{prop}\label{prop:CP}
Suppose that $R_0$ is of type $\textup{C}_n$.
If $X$ is the weight lattice of $R_0$ we denote the corresponding
affine Hecke algebra $\H^{C,e}_n$. In this case we have $[\ga_{W,\Delta}]=1$.
\end{prop}
\begin{proof}
In this case $\H^{C,e}_n$ is simply a specialization of the three
parameter type $\textup{C}^{(1)}_n$ affine Hecke algebra, hence the
result follows from Proposition \ref{prop:BQ}.
\end{proof}
\begin{prop}\label{prop:CQ}
Suppose that $\H$ is the non-extended affine Hecke algebra $\H_n^{C,Q}$
i.e. $R_0$ is of type $\textup{C}_n$ and the lattice $X$ equals the root lattice
of $R_0$. Then $[\ga_{\W,\Delta}]=1$.
\end{prop}
\begin{proof}
Again we compare the situation with the standard case $\textup{C}_n^{(1)}$.
The description of the standard parabolic subsystems $P\subset F_0$ is as before,
where the rightmost factor of type $\textup{B}_l$ has to be replaced by
$\textup{C}_l$ of course.
The new complication is that the affine Hecke algebra $\H_P$ is not always
a tensor product of extended type $\textup{A}$-factors and (possibly) a type
$\textup{C}$ factor. If at least one of the $\lambda_i$ is odd then $\H_P$ is
as before, a tensor product of a number of extended type $\textup{A}$ factors
$\H_{\lambda_{i-1}}^A(\Lambda_{\lambda_{i-1}},q_1)$ with at most one type $\textup{C}$
factor $\H_l^{C,e}$. In other words, the lattice associated with the Hecke algebra
$\H_P$ is the weight lattice $\Lambda_P$ of $R_P$.
However if the $\lambda_i$ are all even,
the algebra $\H_P$ is an index two subalgebra of the affine Hecke algebra just
described obtained by taking the fixed points with respect to the twisting
involution $\psi_\epsilon$ corresponding to the unique $W_P$-invariant quadratic
character $\epsilon$ of the weight lattice $\Lambda_P$ of $R_P$ which is nontrivial
when restricted to any of the weight lattices of the irreducible direct
summands of $R_P$.

In the first case (if there exists at least one odd $\lambda_i$) then the group
$K_P^{C,\Lambda_r}$ is of the form $K_P\times C_2$, where the last factor $C_2$ acts
on the extended algebra $\H_l^{C,e}$ via the twisting automorphism associated
with the nontrivial character of the weight lattice $\Lambda^C$ of type
$\textup{C}_l$ trivial on the root lattice $\Lambda^C_r$. The argument to see the
triviality of $\ga_{\W,\Delta}$ is now exactly analogous to the proof of
Proposition \ref{prop:BP}.

In the second
case (when all $\lambda_i$ are even) then $K_P^{C,\Lambda_r}$ is the quotient of
the previously described group by the subgroup $K_\epsilon=\langle\epsilon\rangle$
generated by $\epsilon$.
The second case can be reduced to the first case as follows.
Let $\H_P^e$ be the semisimple affine Hecke algebra whose root system
is $R_P$ and whose lattice $X_P^e$ equals the weight lattice $\Lambda_P$ of $R_P$.
Then $K_\epsilon$ acts on $\H_P^e$ and $\H_P=(\H_P^e)^{K_\epsilon}$. We may assume
that $P\not= F_0$, otherwise there is nothing to prove. But then, by definition,
$X_P$ contains factors of the form $\Lambda_{\lambda_i-1}^A$ with
$\lambda_i\geq 2$. Thus $\epsilon$ is a $W_P$-invariant character of $X_P^e$,
and is nontrivial on each of the type $\textup{A}$ factors.
Now we use the following special feature of the affine Hecke algebras
$\H_{\lambda_i-1}^{A,X^A}$ of type $\textup{A}_{\lambda_i-1}$ (and any lattice
$X^A_{\lambda_i-1}$): For such affine Hecke algebras, twisting by a nontrivial
$W_{\lambda_i-1,0}^{A}$-invariant character $k\in T^A_{\lambda_i-1}$ has no
fixed points on the set of equivalence classes of discrete series characters
$\Delta_{\lambda_i-1}^A$ (this follows from considering the unitary part of
the central characters of the discrete series characters, as in the proof of
Proposition \ref{prop:BQ}).
Since $\H_P^e$ is merely a product of such type $\textup{A}$-factors and a
type $\textup{C}$ factor, this shows in view of the above that
``twisting by $\epsilon$'' acts on the set $\Delta_P^e$ of equivalence classes
of discrete series representations of $\H_P^e$ \emph{without
fixed points}. In turn this implies (by elementary Clifford
theory, see \cite{RamRam}, and Lemma \ref{lem:ind})
that the restriction functor sends irreducible discrete series of $\H_P^e$
to irreducible discrete series of $\H_P$, and all discrete series of $\H_P$
are obtained in this way. Hence the action groupoid $\W^e_{P,\Delta_P^e}$
of the group $K^e_P\rtimes \Wf_{P,P}$ acting on the set of equivalence
classes $\Delta_P^e$ of irreducible discrete series representations of $\H_P^e$
via twisting automorphisms on $\H_P^e$, is Morita equivalent to $\W_{P,\Delta_P}$.
But $\W_{P,\Delta_P^e}^e$ is a union of groups
$\W^e_{(P,\tilde{\delta}_P),(P,\tilde{\delta}_P)}$ which are of the same form
as in the first case, reducing the second case to the first case as required.
\end{proof}
\subsection{$R_0$ of type $\textup{D}_n$ and $X\not=\Lambda_r$}
\begin{prop}\label{prop:Drest}
Suppose that $\H$ is an affine Hecke algebra of type
$\textup{D}_n$ whose lattice $X$ is not the root lattice.
Then $[\ga_{\W,\Delta}]=1$.
\end{prop}
\begin{proof}
Again the standard parabolic subsystems $R_P$ of $R_0$ are
products of a number of type $\textup{A}$ factors and at most
one type $\textup{D}$ factor. The complicating aspect in the present
case is the more complicated structure of the group of automorphisms of the
type $\textup{D}$ factor that needs to be considered.

First let us suppose that
the lattice $X=\Lambda$, the weight lattice of $R_0$.
In this case the description of the group $K_P^{D,\Lambda}$
similar to $K_P^{B,\Lambda}$ as in the proof of Proposition \ref{prop:BP},
where we may and will assume now that $4\leq l<n$, since otherwise $R_P$
either equals $R_0$ (and there is nothing to prove in this case)
or has only type $\textup{A}$-factors (which brings us to the situation
of Proposition \ref{prop:A}).
We write
\begin{equation}
\Lambda=\Lambda^D_n=\Lambda_0+\Lambda_1=\mathbb{Z}^n+((1/2,1/2,\dots,1/2)+\mathbb{Z}^n)
\end{equation}
If there exists at least one odd $\lambda_i$
then $K_P^{D,\Lambda}$
is a direct product of the cyclic groups corresponding
to the type $\textup{A}$-factors (hence the corresponding twisting
automorphisms act trivially on the type $\textup{D}$-factor) and a
factor $C_2=\langle\eta\rangle$ acting nontrivially only on the
type $\textup{D}$-factor, where $\eta\in K_l^{D}(\Lambda^D_l)$ is
the unique character of $\Lambda^D_l$ with kernel $\mathbb{Z}^l$
(acting by twisting on the tensor factor $\H_l^D(\Lambda)$ of $\H_P$).
If all $\lambda_i$ are even, all the direct factors of the group
just described are even cyclic groups, and thus carry a unique
nontrivial quadratic character. Then $K_P^{D,\Lambda}$ is
the kernel of the product $\rho$ of all these nontrivial
characters on the cyclic factors.

Let $M$ be the largest part
of $\lambda=(\lambda_1,\dots,\lambda_r)\vdash n-l$, and let
$\mu_i$ (for $i=1,\dots, M$) be the multiplicity of
$i$ as a part of $\lambda$.
Then it is easy to see that
\begin{equation}\label{eq:Wf}
\Wf_{P,P}=(W_0(\textup{B}_{\mu_1})\times\dots\times W_0(\textup{B}_{\mu_M})\times \langle\omega\rangle)^{\Sigma_{odd}}
\end{equation}
where $\omega$ is the restriction of the unique nontrivial diagram automorphism of $\textup{D}_n$
to the subdiagram of type $\textup{D}_l$, and
where ${\Sigma_{odd}}$ is the product over all factors $W_0(\textup{B}_{\mu_{2j+1}})$ of the
linear character $\Sigma_{2j+1}$ given by taking the product of the signs
of a signed permutation and, in the last factor, of the unique nontrivial character of
$\langle\omega\rangle\approx C_2$.

In all cases we define an extension
$\W_{P,P}^e\supset \W_{P,P}$ of order $2$, where $\W_{P,P}^e$ is of the form
\begin{equation}
\W_{P,P}^e=\W_{P,P}^{A,e}\times \W_{P,P}^{D,e}
\end{equation}
with
\begin{equation}
\W_{P,P}^{D,e}=\langle\omega\rangle\times\langle\eta\rangle\approx C_2\times C_2
\end{equation}
if $\lambda$ contains odd parts, and
\begin{equation}
\W_{P,P}^{D,e}=\langle\eta\rangle\approx C_2
\end{equation}
if $\lambda$ has only even parts.
In the latter case it is again clear that $[\ga_{(P,\delta_P)}]=1$
as in Proposition \ref{prop:BP}. In the first case we need to show that
if a discrete series representation $\delta_P^D$ of $\H^D_l(\Lambda^D_l)$
contains $\langle\omega\rangle\times\langle\eta\rangle$ in its isotropy group
then it can be extended to a representation of
$\H^D_l(\Lambda^D_l)\rtimes(\langle\omega\rangle\times\langle\eta\rangle)$.
This follows from Lemma \ref{lem:maxext}.

Next we assume that $X=\mathbb{Z}^n$. With the previous situation $X=\Lambda$
in mind this case is easier, since everything is the same except that $K_P$ does not
involve the factor $\langle\eta\rangle$ now (compare with the proof of
Proposition \ref{prop:BQ}). The above arguments apply in this simpler situation as well
(but the extension of $\delta^D_P$ is obvious now,
since its isotropy is at most a $C_2$) showing that $[\ga]=1$ in this case.

Finally if $n$ is even, we need to consider two more lattices
$X=\Lambda_-$ and $\Lambda_+$ with
\begin{equation}
\Lambda_\pm=\Lambda_r+((1/2,1/2,\dots,\pm 1/2)+\Lambda_r)
\end{equation}
Let $\Lambda_P=\Lambda^A_\lambda\times \Lambda^D_l$ be the weight
lattice of $R_P$. Observe that the second projection of
$X_P\subset \Lambda_P$
is equal to the full weight lattice $\Lambda^D_l$ of type $\textup{D}_l$
(unless $l=n$, a case which we excluded at the start of this proof).
Let $K_P^\Lambda$ be the subgroup of the group of characters of the
weight lattice $\Lambda_P$ which
restrict to $1$ on the sublattice $X\cap \mathbb{R}P$. Restriction
of characters of $\Lambda_P$ to $X_P$ induces a quotient map
$q:K_P^\Lambda\mapsto K_P$. The above observation implies that
the first projection $p_1:K_P^\Lambda\to K_\lambda^A$ (by restriction
of a characters of $\Lambda_P$ to the factor $\Lambda_\lambda^A$)
is injective on the kernel of $q$. Hence, the arguments in the proof
of Proposition \ref{prop:CQ} (reducing the ``second case'' to the ``first case'')
apply and show that we can replace the quotient $K_P$ by $K_P^\Lambda$
via a suitable equivalence of groupoids. In this situation we may write
$\delta_P^e=\delta^{A,e}_P\otimes\delta^{D,e}_P$ for the extension
to $\H_P^e$ of the discrete series representation $\delta_P$ of $\H_P$. The group
$\W_{P,P}^e=K_P^e\rtimes\Wf_{P,P}^e$ of automorphisms of $\H_P^e$ that
needs to be considered now is always a subgroup
(depending on $P$) of the automorphism group
$\W_{P,P}^m=\W_{P,P}^{A,m}\times\W_{P,P}^{D,m}$ of $\H_P^e$
described by
\begin{equation}\label{eq:WPA}
\W_{P,P}^{A,m}=K_P^{A,m}\rtimes \Wf_{P,P}^A
\end{equation}
with
\begin{equation}\label{eq:WfP}
\Wf_{P,P}^A=W_0(\textup{B}_{\mu_1})\times\dots\times W_0(\textup{B}_{\mu_M});
K_P^{A,m}=C_2^{\mu_2}\times\dots\times C_M^{\mu_M},
\end{equation}
and with
\begin{equation}\label{eq:WPD}
\W_{P,P}^{D,m}=K_P^{D,m}\rtimes\langle\omega\rangle
\end{equation}
where $K_P^{D,m}$ is the character group of $\Lambda^D_l/\Lambda^D_{r,l}$
(a group of order $4$).
Therefore, arguing as in the proof of Proposition \ref{prop:BP}, it suffices
to show that $\delta_P^{D,e}$ extends to a linear representation of
$\H_P^{D,e}\rtimes I_{(P,\delta^{D,e}_P)}$, where
$I_{(P,\delta^{D,e}_P)}$ is the isotropy group of
$[\delta_P^{D,e}]$ in $\W_{P,P}^{D,m}$. This follows from
Lemma \ref{lem:maxext}, finishing the proof.
\end{proof}
The following lemma uses a nontrivial property of irreducible
discrete series representations of the graded affine Hecke algebra type
$\textup{D}_n$ proved in \cite{OpSo}.
\begin{lem}\label{lem:maxext}
Let $\H$ be the affine Hecke algebra of type $\textup{D}_l$ (with $l\geq 4$)
which is maximally extended, i.e. $X=\Lambda^D_l$, the weight lattice of $R_0$.
For convenience we take the standard realization of $R_0$, with
basis $\{e_1-e_2,\dots,e_{n-1}-e_n,e_{n-1}+e_n\}$.
Let $K$ be the group of characters of $\Lambda^D_l/\Lambda^D_{r,l}$ (a group of
order $4$) and let $\omega$ be the diagram automorphism of $R_0$ induced by the
orthogonal reflection in the hyperplane $x_n=0$. We let the group
$\W:=K\rtimes\langle\omega\rangle$ act on $\H$ by twisting automorphisms as usual.
Let $\delta\in\Delta(\H)$, and let $I_\delta$ be the isotropy group
of $[\delta]$ in $\W$. Then $\delta$ extends to a representation of $\H\rtimes I_\delta$.
\end{lem}
\begin{proof}
We first observe that $\W\approx \mathbb{D}_8$, the dihedral group of order $8$.
Let $\epsilon\in\W$ be an element of order $4$, and let $\eta=\epsilon^2$
(the generator of the center of $\W$). Then $\W=\langle\epsilon,\omega\rangle$.
We define $\kappa:=\omega\epsilon$, an element of order $2$. There are $3$ subgroups
of index $2$ in $\mathbb{D}_8$, one of which is cyclic.
If $l$ is odd then $K=\langle\epsilon\rangle\approx C_4$;
if $l$ is even then $K=\langle\eta,\kappa\rangle\approx C_2\times C_2$.
The other subgroup of index $2$ is $N=\langle\eta,\omega\rangle\approx C_2\times C_2$.

It follows from  \cite[Theorem 7.1, Theorem 8.7]{OpSo}
that any $\delta\in\Delta(\H)$ admits (precisely two) extensions
($\delta_-,\delta_+$ say)
to $\H\rtimes\langle\omega\rangle=\H^B_l(\Lambda^B_l)$
(with $q_{2x_n}=1$).
In particular, we always have $\omega\in I_\delta$.
It follows that either $I_\delta$ is cyclic (in which case
the desired result is obvious, since cyclic groups have a trivial
Schur multiplier (see the remark at the end of the proof of
Theorem 6.5)) or $N\subset I_\delta$.

We now use \cite{IY} that the Schur multiplier of the group $\mathbb{D}_8$ is
$C_2$, and that \cite[Section 3]{R} the restriction of its unique nontrivial
class $[\alpha]\in H^2(\mathbb{D}_8,\mathbb{C}^\times)$ to both subgroups of type
$C_2\times C_2$ is nontrivial. Combined with the above remarks we see that it
suffices to prove that $\delta$ extends to $\H\rtimes N$ if $N\subset I_\delta$,
which is what we will assume from now on.

As we have already remarked, $\delta$ extends
to irreducible discrete series representations $\delta_\pm$ of
$\H^B_l(\Lambda^B_l)$ (with $q_{2x_n}=1$).
This algebra admits an involutive automorphism $\eta^B$
whose fixed point set is $\H^B_l(\mathbb{Z}^l)$
(with $q_0=q_2=1$; recall that $\Lambda^B_{r,l}=\mathbb{Z}^l$)
which fixes $\omega$ and which restricts to $\eta$ on
$\H^D_l(\Lambda^D_l)\subset\H^B_l(\Lambda^D_l)$.
To prove that $\delta$ extends to $\H\rtimes N$ it suffices to show that $\eta^B$
is in the isotropy group of $[\delta_\pm]$. We claim actually that this holds
true for any $\delta\in\Delta(\H^B_l(\Lambda^B_l))$.
Indeed, applying \cite[Theorem 7.1, Theorem 8.7]{OpSo} to
$\H^B_l(\Lambda^B_l)$ (with $q_{2x_n}=1$) and to its fixed point algebra
$\H^B_l(\mathbb{Z}^l)=\H(\textup{C}_l^{(1)})$ (with $q_0=q_2$)
(see \cite[Example 8.3]{OpSo}) under $\eta^B$ we see that
$|\Delta^B_l(\mathbb{Z}^l)|=2|\Delta^B_l(\Lambda^B_l)|$. In view of
Lemma \ref{lem:ind} and \cite[Theorem A.13]{RamRam} we conclude on the other
hand that if there would exist a $\delta\in\H^B_l(\Lambda^B_l)$ whose class is not
invariant for $\eta^B$ then we would necessarily have
$|\Delta^B_l(\mathbb{Z}^l)|<2|\Delta^B_l(\Lambda^B_l)|$, proving our claim and thus
finishing the proof of the Lemma.
\end{proof}
\subsection{$R_0$ of type $\textup{D}_n$ and $X=\Lambda_r$}
In the one remaining classical case, the affine Hecke algebra of type $\textup{D}_n$
with $X=\Lambda^D_{r,n}$, the root lattice of $R_0$, it is (remarkably) not always true that
that $[\ga_{\W,\Delta}]=1$. Yet we have:
\begin{prop}\label{prop:DQ}
Let $\H=\H^D_n(\Lambda_{r,n})$ and let $\W_{\Xi_u}$ be its groupoid of unitary standard
induction data. Then $[\ga_{\W,\Xi}]=1$.
\end{prop}
\begin{proof}
Let $\xi=(P,\delta,t)\in\Xi_u$. We need to show that the $2$-cocycle $\ga_\xi$ of
$\W_{\xi,\xi}$ is a coboundary.
Let $P$ be as in (\ref{eq:Atype}) with $4\leq l<n$ and as before, let $\mu_i$ denote
the multiplicity of the part $i$ in $\lambda\vdash n-l$.
The group $\Wf_{P,P}$ does not depend on the lattice $X$, so is still given
by (\ref{eq:Wf}). Let us write (\ref{eq:Wf}) as
\begin{equation}
\Wf_{P,P}=(\Wf_{P,P}^A\times\Wf_{P,P}^D)^{\Sigma_{odd}}
\end{equation}
with $\Wf_{P,P}^D=\langle\omega\rangle$ ($\omega$ being the unique nontrivial
automorphism of the diagram of type $\textup{D}_l$ that extends to $\textup{D}_n$)
and
\begin{equation}\label{eq:WfPP}
\Wf_{P,P}^A=W_0(\textup{B}_{\mu_1})\times\dots\times W_0(\textup{B}_{\mu_M})
\end{equation}
and ${\Sigma_{odd}}$ the linear character defined in the text just below (\ref{eq:Wf}).
Recall that ${\Sigma_{odd}}$ is trivial on
$\Wf_{P,P}^A$ iff all parts of $\lambda$ are even.
We introduce the projections
\begin{equation}
\pi_\Wf^A:\Wf_{P,P}\to\Wf^A_{P,P},\quad \pi_\Wf^D:\Wf_{P,P}\to\Wf^D_{P,P}
\end{equation}
Notice that $\pi_\Wf^A$ is an isomorphism (always) and $\pi^D_\Wf$ is trivial
iff all parts of $\lambda$ are even.
We have by definition
\begin{equation}\label{eq:wpp}
\W_{P,P}=K_P\rtimes \Wf_{P,P}
\end{equation}
Let us now compute the lattice $X_P$ and the abelian group $K_P$ (this is similar to
the proof of Proposition \ref{prop:CQ}).
The orthogonal projection $X_P$ of the root lattice $\Lambda_{r,n}^D$ onto
$\mathbb{R}R_P$ is the product of the weight lattice of the
type $\textup{A}$ factors of $R_P$ with the lattice $\mathbb{Z}^l$
for the type $\textup{D}_l$ factor of $R_P$, provided that $\lambda$ has odd parts.
Hence in this case we have
\begin{equation}\label{eq:HPtensor}
\H_P=\H_P^A(\Lambda_P)\otimes\H^D_l(\mathbb{Z}^l)
\end{equation}
If $\lambda$ has only even parts then $X_P$ is a sublattice of index two of the lattice just
described, namely the kernel of the product $\epsilon$ of the unique nontrivial
$W_P$-invariant quadratic character of the direct summands of the above lattice
corresponding to the irreducible components of $R_P$.
Hence if $\lambda$ has odd parts then
\begin{equation}\label{eq:kappagen}
K_P=K_P^A\times K_P^D
\end{equation}
where
\begin{equation}\label{eq:KPA}
K_P^A\approx C_2^{\mu_2}\times\dots\times C_M^{\mu_M}
\end{equation}
and where
\begin{equation}\label{eq:KPD}
K_P^D:=\langle\kappa\rangle
\end{equation}
with $\kappa$ the unique nontrivial character of the lattice
$\mathbb{Z}^l$ which is trivial on the root lattice $\Lambda_{r,l}^D$ of
the root system $\textup{D}_l$. In this case we denote by
\begin{equation}
\pi_K^A:K_P\to K_P^A,\quad\pi_K^D:K_P\to K_P^D
\end{equation}
the projections onto the type $\textup{A}$ and the type $\textup{D}$ factors.
We remark that $\pi^D_K$ is invariant for the action of $\W_{P,P}$ on $K_P$.
If $\lambda$ is even then $K_P$ is the quotient
of this group by $\langle\epsilon\rangle$.
If $\lambda$ has odd parts then we define
\begin{equation}\label{eq:piD}
\pi^D=\pi_\Wf^D\times\pi_K^D:\W_{P,P}\to \W_{P,P}^D:=\Wf^D_{P,P}\times K_P^D\subset\textup{Aut}(\H^D_l(\mathbb{Z}^l))
\end{equation}
Observe that $\W_{P,P}^D=\langle\omega\rangle\times\langle\kappa\rangle \approx C_2\times C_2$.
Similarly we define
\begin{equation}
\pi^A=\pi_\Wf^A\times\pi_K^A:\W_{P,P}\to \W_{P,P}^A:=\Wf^A_{P,P}\ltimes K_P^A\subset\textup{Aut}(\H^A_P(\Lambda_P))
\end{equation}
If $\lambda$ is even, then (as in the proof of Proposition \ref{prop:CQ}) we
extend the lattice $X_P$ of $\H_P$ to obtain $\H_P^e\supset\H_P$
(a quadratic extension) whose associated
lattice $X_P^e$ is the product of the weight lattices of the
type $\textup{A}$ factors of $R_P$ with the lattice $\mathbb{Z}^l$
for the type $\textup{D}_l$ factor of $R_P$.
As in Proposition \ref{prop:CQ} this leads
to a groupoid $\W^e_{P,\Delta_P^e}$ which is Morita equivalent to $\W_{P,\Delta}$.
Using that $\pi_\Wf^D=1$ if $\lambda$ is even we can now apply the same argument
as given in the proof of Proposition \ref{prop:CQ} to conclude that
$\gamma_{P,\Delta^e}^e$ is trivial. As we know this implies the triviality of
$\ga_{\W_{P,P},\Xi_P}$ as well in this case.

\emph{So let us assume from now that $\lambda$ contains at least one odd part.}
To prove
the triviality of $\ga_{\W_{P,P},\Xi_P}$ it is enough to prove the triviality
of $\ga_\xi$ for the isotopy group $\W_{\xi,\xi}$ of an arbitrary
object $\xi=(P,\delta_P,t)$ with
$\delta_P=\delta^{A_{\lambda_1-1}}\otimes\dots\otimes\delta^{A_{\lambda_r-1}}\otimes\delta$.
Let us therefore consider the action of $\W_{P,P}$ on the space $\Xi_P$ of parameters first.
Using the action of $K_P^A$ we may and will assume that the one dimensional type $\textup{A}$
discrete series representations $\delta^{A_{\lambda_i-1}}$ all have a real central character.

In the present situation we have $T=(\mathbb{C}^\times)^n/\langle\pm 1\rangle$.
Recall that $T^P\subset T$ is the subtorus of the characters of the orthogonal projection of
$X=\Lambda_{r,n}^D$ onto the subspace $\mathbb{R}P^\perp$.
For each part $i$ of $\lambda$ this projection is generated
by generators $E^i_1,\dots,E^i_{\mu_i}$ say, which we normalize by requiring that $E^i_j$
has coordinates $1/i$ at the $i$ slots corresponding to the $j$-th part of size $i$ of $\lambda$,
while its remaining coordinates are $0$. Accordingly, for the
element in $t\in T^P$ such that $t(E^i_j)=t^i_j\in\mathbb{C}^\times$ we write:
\begin{equation}\label{eq:coord}
t=(t^1_1,t^1_2,\dots,t^1_{\mu_1},t^2_1,t^2_2,\dots,t^{M-1}_{\mu_{M-1}},t^M_1,t^M_2,\dots,t^M_{\mu_M})
\end{equation}
In order to see how $K_P$ acts on $T^P$ and on $\Delta_P$
we need to identify $K_P=T_P\cap T^P$ as a subgroup
of $T^P$ explicitly.
The group $K_P^A$ is the subgroup of elements $t$ with $(t^i_j)^i=1$,
and acts by multiplication on $T^P$.
However, there exists an additional generator $\kappa\in T^P\cap T_P\subset T$ (since $\lambda$
is not even),
which has its first $n-l$ coordinates (in $T=(\mathbb{C}^\times)^n/\langle\pm 1\rangle$)
equal to $1$ and its last $l$ coordinates equal to $-1$. Indeed, this
description makes it obvious that $\kappa\in T_P$. On the other hand, $\kappa$ can also be
given by a row of $n-l$ coordinates equal to $-1$ and a tail of $l$ coordinates equal to $1$.
This description makes it obvious that $\kappa\in T^P$. Hence we have $\kappa\in T_P\cap T^P$.
By (\ref{eq:kappagen}) $K_P$ is generated by $K_P^A$ and $\kappa$.
In the coordinates (\ref{eq:coord}) on $T^P$ the element $\kappa\in T_P\cap T^P$
equals $-1\in T^P$ (i.e. $\kappa^i_j=-1$ for all $i,j$). The subgroup $K_P\subset T^P$
is thus
given by those $t\in T^P$ with either $(t^i_j)^i=1$ (for all $i$), or with $(t^i_j)^i=-1$ (for all $i$).
The projection $\pi_K^D:K^P\to K_P^D$ is given by $\pi_K^D(k)=\kappa$ if $(k^i_j)^i=-1$ for some
pairs $(i,j)$ (hence all) with $i$ odd, and $\pi^D_K(k)=1$ else. The projection onto $K_P^A$ is
given by $\pi^A_K(k)=k$ if $\pi^D(k)=1$, and $\pi^A_K(k)=-k$ else.

The group $\Wf_{P,P}$ acts on $T^P$ by signed
permutations on the $E^i_j$ which leave the superscript $i$ unchanged.
If $g\in\W_{P,P}$ and $\xi=(P,\delta,t)$ then $g\xi:=(P,\delta^g,gt)$, where
$\delta^g\simeq\delta\circ\phi_g$.
By (\ref{eq:HPtensor}),  $\delta=\delta^A\otimes\delta^D$ and thus
\begin{equation}\label{eq:task}
\delta^g=(\delta^A\otimes\delta^D)^g=(\delta^A)^{\pi^A(g)}\otimes(\delta^D)^{\pi^D(g)}
\end{equation}

Let us show first that $[\ga_\xi]=1$ if the map $\pi^D|_{\W_{\xi,\xi}}\to \W_{P,P}^D$
is not surjective. We extend $\Wf_{P,P}$ to $\Wf^e_{P,P}=\Wf_{P,P}^A\times\Wf_{P,P}^D$ and
accordingly define $\W^e_{P,P}=\W_{P,P}^A\times\W_{P,P}^D$.
By the assumption we see that
$\W_{\xi,\xi}\subset\widetilde{\W}_{P,P}:=\W_{P,P}^A\times C_2$
where $C_2\subset \W_{P,P}^D$ is a suitably chosen subgroup.
But then the projective representation of $\widetilde{\W}_{P,P}$
defines a trivial $2$-cocycle (as in the proof of Proposition \ref{prop:BP},
using also (\ref{eq:HPtensor})),
hence in particular $[\ga_\xi]=1$ in this case.

So let us now assume that $\pi^D|_{\W_{\xi,\xi}}$ is onto $\W_{P,P}^D$.
By (\ref{eq:HPtensor}) this implies in particular that $(\delta^D)^\kappa=\delta^D$
and $(\delta^D)^\omega=\delta^D$.
Thus $\delta^D$ defines a $2$-cocycle $\ga_\delta$ of $\W^D_{P,P}$.
By (\ref{eq:task}) the desired result follows from the claim:
{\emph{The pullback of $\ga_\delta$
under the restriction $\pi^D_{\xi,\xi}:\W_{\xi,\xi}\to\W^D_{P,P}$
of the map $\pi^D$ of (\ref{eq:piD}) is a coboundary}}.
The remaining part of this proof is devoted to the proof of this claim.

The condition for $g=kw\in K_P\rtimes\Wf_{P,P}$ to be in $\W_{\xi,\xi}$
is that $(\delta^A)^{\pi^A(g)}=\delta^A$ and $gt=t$. Since $\Wf^A_{P,P}$ fixes
$\delta^A$ (by our choice to take the central character of $\delta^A$ real)
and $K_P^A$ acts freely on $\Delta_P^A$, the first equation is equivalent to
$\pi^A(k)=1$. Hence $k=\pm 1\in T^P$ (recall that $-1\in T^P$ is the element
$\kappa\in K_P$), and accordingly $\pm wt=t$.
Hence we have
\begin{equation}\label{eq:real}
\W_{\xi,\xi}\approx\{w\in \Wf^A_{P,P}=W_0(B_{\mu_1})\times\dots\times
W_0(B_{\mu_1})\mid wt=\pm t\}
\end{equation}
and in this realization the homomorphism $\pi^D_{\xi,\xi}:\W_{\xi,\xi}\to \W_{P,P}^D$
is given by $\pi_{\xi,\xi}^D(w)=\omega^{\sigma(w)}\times\kappa^{\epsilon(w)}$ where
$\sigma,\epsilon:\W_{\xi,\xi}\to C_2$ are two linear characters defined by
$(-1)^{\sigma(w)}={\Sigma_{odd}}(w)$ (with ${\Sigma_{odd}}$ as in (the proof of)
Proposition \ref{prop:Drest}), and $wt=(-1)^{\epsilon(w)} t$.

The torus $T^P$ is a direct product $T^P=\prod_{i=1}^r T^{P,(i)}$ of the
tori $T^{P,(i)}$ of characters of the root lattice of $\textup{B}_{\mu_i}$.
Now $T^{P,(i)}$ has a double cover $\tilde{T}^{P,(i)}$, the torus of
characters of the weight lattice of $\textup{B}_{\mu_i}$. The kernel of
the covering map is denoted by $\langle\eta^{(i)}\rangle$ where $\eta^{(i)}$
is the unique nontrivial $W_0(\textup{B}_{\mu_i})$-invariant element of
$\tilde{T}^{P,(i)}$. Thus we have
\begin{equation}
1\to\langle\eta^{(i)}\rangle\to\tilde{T}^{P,(i)}\to T^{P,(i)}\to 1
\end{equation}
Putting these together we get an exact sequence
\begin{equation}
1\to\langle\eta^{(1)},\dots,\eta^{(r)}\rangle\to\tilde{T}^P\to T^P\to 1
\end{equation}
where $\langle\eta^{(1)},\dots,\eta^{(r)}\rangle\approx C_2^r$.
By (\ref{eq:real}) we see that the action of $\W_{\xi,\xi}$ on $T^P$
extends to $\tilde{T}^P$.

Consider the set $S_t=\{s\in\tilde{T}^P\mid s\to\{\pm t\}\subset T^P\}$.
This set admits a free, transitive action (by multiplication) of the subgroup
$M:=\langle\eta^{(1)},\dots,\eta^{(r)}\rangle\times\langle\tilde\kappa\rangle$
of $\tilde{T}^P$, where $\tilde\kappa$ denotes a lift of $-1=\kappa\in T^P$.
This abelian group is isomorphic to $C_2^{r-1}\times C_4$ if $\lambda$ contains parts
with an odd multiplicity and is isomorphic to $C_2^{r+1}$ otherwise.
Clearly $M$ is stable for the action of $\W_{\xi,\xi}$ on $\tilde{T}^P$,
making $M$ a module over $\W_{\xi,\xi}$. To describe the module structure
explicitly, remark that the element $\tilde\kappa$ is not
$W_0(B_{\mu_1})\times\dots\times W_0(B_{\mu_1})$-invariant.
In fact, if $w=w^{(1)}\times\dots\times w^{(r)}\in
W_0(B_{\mu_1})\times\dots\times W_0(B_{\mu_1})$ then
\begin{equation}
w(\tilde\kappa)=\eta(w)\tilde\kappa
\end{equation}
where
\begin{equation}
\eta(w)=\prod_{i=1}^r(\eta^{(i)})^{\sigma_i}\in\langle\eta^{(1)},\dots,\eta^{(r)}\rangle
\end{equation}
with $(-1)^{\sigma_i}={\Sigma_{i}}(w^{(i)})$, where $\Sigma_i$ is the character
on $W_0(\textup{B}_{\mu_i})$ whose kernel is $W_0(\textup{D}_{\mu_i})$. The elements
$\eta^{(i)}$ are all fixed for the action of $\W_{\xi,\xi}$.

The $M$-orbit $S_t$ is stable for the action
of $\W_{\xi,\xi}$ on $\tilde{T}^P$ as well. Hence, any lift $\tilde t\in S_t$ of $t$
defines a $1$-cocycle $\mu:\W_{\xi,\xi}\to M$ of $\W_{\xi,\xi}$ with values in $M$
by the formula $w\tilde t=\mu(w)\tilde t$.
We fix such a lift $\tilde t$ once and for all.
Consider the abelian group $N:=\langle\eta,\tilde\kappa\rangle$ defined by the relations
$\eta^2=1$ and $\tilde\kappa^2=\eta$ (if the number of odd $i$ with $\mu_i$ odd is odd)
or else $\tilde\kappa^2=1$ (hence $N$ is isomorphic to $C_4$ if there are an odd number
of odd $i$ with odd $\mu_i$, and $N$ is isomorphic to $C_2\times C_2$ else).
Consider the $\W_{\xi,\xi}$-module structure on $N$ defined by
$w(\eta)=\eta$ for all $w\in\W_{\xi,\xi}$, and by
\begin{equation}\label{eq:actN}
w(\tilde\kappa)=\eta^{\sigma(w)}\tilde\kappa
\end{equation}
(the fact that
this defines a $\W_{\xi,\xi}$-module is equivalent to saying that $\sigma$ is a character of $\W_{\xi,\xi}$).
The module $N$ is a quotient of $M$ via the unique homomorphism $\alpha:M\to N$ satisfying
$\alpha(\tilde\kappa)=\tilde\kappa$, $\alpha(\eta^{(i)})=\eta$ if $i$ is odd, and
$\alpha(\eta^{(i)})=1$ if $i$ is even (in fact, the definition of $N$ is
such that $\alpha$ exists). The cocycle $\mu$ induces a cocycle of $\W_{\xi,\xi}$
with values in $N$ which we denote by $\mu_N$.
Consider the following diagram
\begin{equation}\label{eq:CD}
\xymatrix{
&&N \ar@{_{(}->}[d]_i
&\W_{\xi,\xi}\ar[d]^{\pi^D_{\xi,\xi}}\ar@{.>}[ld]_\chi\\
&\langle\eta\rangle\ar@{^{(}->}[r]
&\mathbb{D}_8\ar@{->>}[r]_j\ar@{->>}[d]
&\W^D_{P,P}\ar@{=}[r]&\langle\omega\rangle\times\langle\kappa\rangle\\
&&\langle\omega\rangle\ar@/_/[u]
}\end{equation}
where $\mathbb{D}_8=\langle\tilde\omega,\tilde\kappa\rangle$ is a dihedral group
of order $8$ in which $\tilde\omega$ has order two and $\eta$
(or more precisely $i(\eta)$) is the nontrivial central element.
The defining relations in $\mathbb{D}_8$ are given by
$\tilde\omega\tilde\kappa\tilde\omega=\eta\tilde\kappa$.
The map $j$ is defined by requesting that $j(\tilde\omega)=\omega$ and
$j(\tilde\kappa)=\kappa$.

We claim that there exists a homomorphism $\chi:\W_{\xi,\xi}\to \mathbb{D}_8$
as indicated in the diagram. Obviously the vertical
exact sequence is split. We choose the splitting $\omega\to\tilde\omega$ of
this sequence. The homomorphism $\pi^D_{\xi,\xi}$ gives rise to
a homomorphism $\pi:\W_{\xi,\xi}\to \mathbb{D}_8$ with image
$\langle\tilde\omega\rangle$ obtained by composing $\pi^D_{\xi,\xi}$ with
the first projection $\langle\omega,\kappa\rangle\to\langle\omega\rangle$
and the lift $\omega\to\tilde\omega$. We can write this explicitly by
$\pi(w)=\tilde\omega^{\sigma(w)}$.
We now define a map $\chi:\W_{\xi,\xi}\to \mathbb{D}_8$ by
\begin{equation}\label{eq:hom}
\chi(w)=i(\mu_N(w))\pi(w).
\end{equation}
We claim that $\chi$ is a group homomorphism. Indeed, the action of $\W_{\xi,\xi}$
on $N$ is related to $\pi$ by the formula
$i(n^w)=\pi(w)i(n)\pi(w)^{-1}$ (using the explicit formulas for the module
$N$ and for $\pi$). It follows that (\ref{eq:hom}) indeed defines a homomorphism.
Next we claim that $\chi$ makes (\ref{eq:CD}) commutative as indicated. Indeed,
$i(\mu_N(w))\equiv \tilde\kappa^{\epsilon(w)}\textup{mod}\langle\eta\rangle$ as
follows from the definition of $\mu$ (and $\mu_N$) and of
$\epsilon$. On the other hand, by construction of $\pi$ we see that
$\pi(w)\equiv\tilde\omega^{\sigma(w)}\textup{mod}\langle\eta\rangle$. Together
these two congruences (modulo the center $\langle\eta\rangle$ of $\mathbb{D}_8$)
imply the claim.

Since $\mathbb{D}_8$ is the Schur extension of $\langle\omega,\kappa\rangle$
it now follows that
the pullback of $\ga_\delta$ under the homomorphism $\pi^D_{\xi,\xi}$
is indeed a coboundary (since $\pi^D_{\xi,\xi}$ factors through the
Schur extension map),
finishing the proof of the claim and of the Theorem.
\end{proof}
\subsection{Final remarks}
\subsubsection{Multiplicity one $W_0$-types}
We would like to comment on a natural alternative approach to proving the
triviality of the cocycles $\ga_\xi$ for $\xi=(P,\d,t)\in\Xi_u$.
The equivalence
class of the restriction to $\H(W_0,q_0)$ of $\pi(\xi)$ is independent of the
continuous parameter $t\in T^P$. We will refer to an irreducible
representation of $\H(W_0,q_0)$ as a ``$W_0$-type'' in this paragraph.
If there exists a $W_0$-type appearing in $(V_\xi,\pi(\xi))$ with multiplicity one then we can
normalize the action of $\mathfrak{R}_\xi$ on $V_\xi$ such that the operators
$\pi(\mathfrak{r},\xi)$ are equal to $1$ on this multiplicity one isotype, and
this trivializes the cocycle $\ga_\xi$.
\begin{prop}
Let $\H=\H(\Ri,q)$ be an affine Hecke algebra, and let $\xi=(P,\d,t)\in\Xi_u$
be a standard tempered induction datum such that the central character of $\d$ is positive
(i.e. infinitesimally real). Then $\ga_\xi$ is trivial.
\end{prop}
\begin{proof}
In view of the above argument, it is sufficient to prove the existence
of a multiplicity one $W_0$-type in $\pi(\xi)$.
We thank Dan Ciubotaru for communicating to us that
it can be shown that any irreducible representation of $\H$
\emph{with positive central character}  admits a multiplicity
one $W_0$-type (see \cite[Introduction, paragraph 1.3]{Chu}).
The proof of this fact is based on case-by-case verifications.
Now consider $\xi=(P,\d,t)$ with $\d$ a discrete series with real central character.
If $t\in T^P$ is positive and sufficiently generic, $\pi(\xi)$ is irreducible
and has positive central character. By Ciubotaru's result mentioned above, this
implies the existence of a multiplicity one $W_0$-type. As explained above,
it follows that $\pi(\xi)$ has a multiplicity one $W_0$-type for all $t\in T^P$,
hence in particular for all $t\in T^P_u$ as desired.
\end{proof}
We do not know how to generalize this argument to general $\d$.
\subsubsection{Examples where $\ga_\Delta$ is nontrivial}\label{subsub:nontriv}
In this subsection we present an example showing
that $\ga_\Delta$ is not trivial for $\H_n^D(\Lambda_{r,n})$ if $n>8$.
In the notation of the proof of Proposition \ref{prop:DQ},
we write $l=n-1$ if $n$ is odd, and $l=n-2$ is $n$ is even, and put $l=2m$
in both cases. We define $\lambda=(1)$ or $\lambda=(1,1)$ depending on
$n$ being odd or even. Recall that conjugate partition of $\lambda$
is denoted by $\mu$; hence we have $\mu=(1)$ in the first case
and $\mu=(2)$ in the second case.
By (\ref{eq:Wf}), (\ref{eq:wpp}), (\ref{eq:KPA}) and (\ref{eq:KPD})
we see that
\begin{equation}\label{eq:wppD}
\W_{P,P}=(W_0(B_{\mu_1})\times\langle\kappa\rangle\times\langle
\omega\rangle)^{\Sigma_{odd}}\subset \W_{P,P}^A\times\W_{P,P}^D
\approx W_0(B_{\mu_1})\times C_2^2
\end{equation}
where $\omega\in\Wf_{P,P}^D$ and $\kappa\in K^D_P$ are as in the proof of
Proposition \ref{prop:DQ}, and
where $\Sigma_{odd}$ is the linear character which is equal to the product
of the signs of a signed permutation in $W_0(B_{\mu_1})$, which is trivial
in $\langle \kappa\rangle$, and which is nontrivial on $\langle \omega\rangle$.
In particular the homomorphism $\pi^D:\W_{P,P}\to\W_{P,P}^D$ of (\ref{eq:piD})
is surjective (even has a section), with
\begin{equation}
\W_{P,P}^D=\langle \kappa\rangle\times\langle\omega\rangle\approx C_2^2
\end{equation}
Equation (\ref{eq:HPtensor})
reduces in this case to
\begin{equation}
\H_P=\H_{2m}^D(\mathbb{Z}^{2m})
\end{equation}
and the action of $\W_{P,P}$ by automorphisms of $\H_P$ factors through
the surjective projection $\pi^D$ to an action (denoted by $\beta$) of $\W_{P,P}^D$
on $\H_P$ by automorphisms.

The spectral diagram
(in the sense of \cite[Definition 8.1]{OpSo}) of $\H_P$ is the affine Dynkin diagram
of type $D_{2m}$ equipped with the action of the unique nontrivial diagram automorphism
$\eta$ whose set of fixed points is the set of non-extremal vertices of the
diagram. In Figure \ref{D2n} this spectral diagram is displayed, with the action
of $\eta$ indicated by the solid arrows. In addition we have indicated in Figure \ref{D2n}
the middle vertex $e$ (the encircled vertex) of the diagram, and the action of three diagram
automorphisms $\tilde\kappa$, $\tilde\omega$ and $\tilde\omega^\prime$ (by the dashed arrows).
The group $G$ of diagram automorphisms generated by $\tilde\kappa$ and $\tilde\omega$ is
isomorphic to the dihedral group $\mathbb{D}_8$, and we have a projection $q:G\to\W_{P,P}^D$
with kernel $\langle\eta\rangle$ (the center of $G$). Observe
that $\eta=\tilde\omega^\prime\tilde\omega=\tilde\kappa\tilde\omega\tilde\kappa\tilde\omega$
in $G$.
\input{D2n.TpX}
Recall that $T_P=\operatorname{Hom}(\mathbb{Z}^{2m},\mathbb{C}^\times)=
(\mathbb{C}^\times)^{2m}$, and denote by $\tilde{T}_P$ the double cover
$\tilde{T}_P:=\operatorname{Hom}(\Lambda_{2m}^D,\mathbb{C}^\times)$ of $T_P$.
There is a canonical identification of the group of special diagram automorphisms
$\langle\tilde\kappa\rangle\times\langle\eta\rangle$ with the group
$\tilde{K}_P:=(\Lambda_{r,2m}^D)^\vee/(\Lambda_{2m}^D)^\vee\approx C_2^2\subset\tilde{T}^P$
of fixed points for natural action of $W_0(D_{2m})$ on $\tilde{T}_P$.
We can extend the action of $W_0(D_{2m})$ on $\tilde{T}_P$ to
$W_0(B_{2m})=W_0(D_{2m})\rtimes\langle\tilde\omega\rangle$.
Then $\eta$ is the unique nontrivial fixed point for
this action of $W_0(B_{2m})$, while $\tilde\omega(\tilde\kappa)=\eta\tilde\kappa$.
We have $T_P=\tilde{T}_P/\langle\eta\rangle$, and the action of $W_0(B_{2m})$ on $T_P$
admits a unique nontrivial fixed point
$\kappa=(-1,\dots,-1)=\tilde\kappa\langle\eta\rangle$.
The group $G=\tilde{K}_P\rtimes\langle\tilde\omega\rangle$ acts naturally on
$\H^D_{2m}(\Lambda_{2m})$ via an action $\tilde\beta$ defined as follows:
$\tilde\beta(\tilde\omega)$ is the diagram automorphism arising from the automorphism
of the root datum of $\H^D_{2m}(\Lambda_{2m})$ which $\tilde\omega$ induces,
and $k\in \tilde{K}_P$ acts on the Bernstein basis of $\H^D_{2m}(\Lambda_{2m})$ by
$\tilde{\beta}(k)(\theta_xN_w)=k(x)\theta_xN_w$.
Since $\eta$ is central in $G$, we see that $\H_P=(\H^D_{2m}(\Lambda_{2m}))^\eta$
is stable for the action of $G$ via $\tilde\beta$. If we restrict the action of
$\tilde\beta$ to the subalgebra $\H_P$, this restricted action descends to
an action of $\W_{P,P}^D$ on $\H_P$ which coincides with
the action $\beta$ on $\H_P$ defined above.

Recall Lusztig's parameterization \cite{Lu},\cite[Theorem 8.7]{OpSo} of the discrete series
representations of $\H_P=\H_{2m}^D(\mathbb{Z}^{2m})$. According to this result
the discrete series representations of $\H_P$ whose central character
$W_0(D_{2m})r\subset \operatorname{Hom}(\mathbb{Z}^{2m},\mathbb{C}^\times)$ contains points
with unitary part equal to $s(e)=(1,\dots,1,-1,\dots,-1)\in T_P$
(with the same number $m$ of $1$'s and $-1$'s) correspond to discrete series
representations of the extended graded affine Hecke algebra of the form
\begin{equation}\label{eq:He}
\mathbf{H}_e:=(\mathbf{H}(D_m)\otimes\mathbf{H}(D_m))\rtimes\langle\eta\rangle
\end{equation}
where $\mathbf{H}(D_m)$ is shorthand for $\mathbf{H}(R_1,V,F_1,k)$
(in the notation of \cite{OpSo}), where $R_1$ is a root system of type $D_m$ in
$V^*$, with basis of simple roots $F_1$.
The underlying based root system of the Hecke algebra $\mathbf{H}_e$ is
the root system $R_{s(e),1}$ of type $D_m\times D_m$ obtained from the spectral
diagram by deleting the vertex $e$. The group $G=\mathbb{D}_8$ of diagram automorphisms fixes $e$;
this yields an action (denoted by $\alpha$) of $G$ on the algebra $\mathbf{H}_e$
by diagram automorphisms. Observe that $\alpha(\eta)$ is inner on $\mathbf{H}_e$, hence
$\alpha$ gives rise to a homomorphism of $G$ to the group of outer automorphisms
of $\mathbf{H}_e$ which factors through $\W_{P,P}^D$ via $q$.

To explain the above mentioned correspondence between the discrete series on
both sides, recall from \cite[proof of Theorem 7.1]{OpSo} that
every discrete series representation $\delta$ of $\mathbf{H}(D_m)$ is fixed for twisting
with the action of the unique nontrivial diagram automorphism $\tilde\omega$ of $D_m$, and can
in fact be extended to
$\mathbf{H}(B_m)=\mathbf{H}(D_m)\rtimes\langle\tilde\omega\rangle$ in precisely two ways,
denoted by $\delta_+$ and $\delta_-$.
In particular, the central character $W_0(D_m)\xi$ of
$\delta$ is fixed for the action of $\tilde\omega$.
If $\delta_+(\tilde\omega)=\Omega\in\operatorname{GL}(V_\delta)$ then
$\delta_-(\tilde\omega)=-\Omega$.
A discrete series representation of $\mathbf{H}_e$ is therefore of the form
$\delta_{1,\epsilon_1}\otimes\delta_{2,\epsilon_2}$ (with $\epsilon_i=\pm$, and
where $\delta_i$ is a discrete series representations of $\mathbf{H}(D_m)$) and
has a central character of the form
$c_V:=((W_0(D_m)\times W_0(D_m))\rtimes\langle\eta\rangle)(\xi_1,\xi_2)=
(W_0(D_m)(\xi_1), W_0(D_m)(\xi_2))$. Notice that the identity map defines an
equivalence of $\mathbf{H}_e$-representations
\begin{equation}
\delta_{1,\epsilon_1}\otimes\delta_{2,\epsilon_2}\isom
\delta_{1,-\epsilon_1}\otimes\delta_{2,-\epsilon_2}
\end{equation}
This representation corresponds (in the sense of \cite[Theorem 8.7]{OpSo})
to a discrete series representation
$\sigma$ of $\H_P$ with central character
$cc:=W_0(D_{2m})(r_\sigma)$ with $r_\sigma:=s(e)\exp{(\xi_1,\xi_2)}$.
We note that $c:=s(e)\exp{c_V}=s(e)(W_0(D_m)(\exp(\xi_1)), W_0(D_m)(\exp(\xi_2)))
\subset cc$ is an equivalence class in the sense of \cite[paragraph 8.1]{Lu}.
The correspondence discussed above is completely determined, using Lusztig's isomorphism
\cite[Section 8, Section 9]{Lu}
\begin{equation}\label{eq:lusisombas}
\Phi:e_c(\bar{\H}_{P,cc})e_c\isom\bar{\mathbf{H}}_{e,c_V},
\end{equation}
by the requirement that the action of $\mathbf{H}_e$ on
$\sigma(e_c)(V_\sigma)\subset V_\sigma$ via $\Phi$ is equivalent to the
representation $\delta_{1,\epsilon_1}\otimes\delta_{2,\epsilon_2}$. It is clear
by the above that these representations of $\mathbf{H}_e$ are invariant for
twisting by $\tilde\omega$, and that they are invariant for $\tilde\kappa$ if and only if
$\delta_1=\delta_2:=\delta$ (a discrete series representation of
$\mathbf{H}(D_m)$). Write $\sigma(\delta)$ for the discrete series
representation of $\H_P$ corresponding to $\delta_+\otimes\delta_+$ and
$\bar\sigma(\delta)$ for the one corresponding to
$\delta_+\otimes\delta_-$.
\begin{prop}
The discrete series representations of $\H_P$ of the form $\bar\sigma(\delta)$
are invariant for twisting by the group
$\W_{P,P}^D=\langle \kappa\rangle\times\langle\omega\rangle\approx C_2^2$ of
automorphisms of $\H_P$, and the corresponding factor set $\ga$
(see \cite[8.32]{CR}) is a nontrivial cocycle of $\W_{P,P}^D$.
\end{prop}
\begin{proof}
The invariance of $\bar\sigma(\delta)$ was discussed above.
We trace the action $\beta$ of $\W_{P,P}^D$ on $\bar{\H}_{P,cc}$ through
Lusztig's isomorphism (\ref{eq:lusisombas}). But since the equivalence class
$c$ is $\omega$-invariant
but not $\kappa$ invariant we are forced to work with the $\W_{P,P}^D$-invariant
idempotent $e_c+e_{\kappa(c)}$ rather than $e_c$. Following Lusztig \cite{Lu}, we
choose $w\in W_0(D_{2m})$ such that $\kappa(c)=wc$ and such that $w$ has minimal
length (or equivalently, such that $w(F_{s(e),1})=F_{s(e),1}$). We have an algebra
isomorphism
\begin{align}\label{eq:lusisom}
\Psi:(e_c+e_{wc})\bar{\H}_{P,cc}(e_c+e_{wc})&\isom
\operatorname{Mat}_{2\times 2}(\bar{\mathbf{H}}_{e,c_V})\\
\nonumber X&\longrightarrow
\begin{pmatrix}
  \Phi(e_cXe_c) & \Phi(e_cXe_{wc}i^0_w) \\
  \Phi(i^0_{w^{-1}}e_{wc}Xe_c) & \Phi(i^0_{w^{-1}}e_{wc}Xe_{wc}i^0_w)
\end{pmatrix}
\end{align}
We transfer the action $\beta$ of $\W_{P,P}^D$ on $(e_c+e_{wc})\bar{\H}_{P,cc}(e_c+e_{wc})$
to the matrix algebra on the right hand side via $\Psi$; we shall denote the resulting
action of $\W_{P,P}^D$ by $\mu$.
We use the isomorphism $\operatorname{Mat}_{2\times 2}(\bar{\mathbf{H}}_{e,c_V})
=\operatorname{Mat}_{2\times 2}(\mathbb{C})\otimes\bar{\mathbf{H}}_{e,c_V}$, and
write $I$ for the identity automorphism of
$\operatorname{Mat}_{2\times 2}(\mathbb{C})$.
Using the results of \cite[Section 8]{Lu} it is not difficult to show that
\begin{equation}
\mu(\omega):=\Psi\circ\beta(\omega)\circ\Psi^{-1}=
C_{\begin{pmatrix}
  1 & 0 \\
  0 & \eta
\end{pmatrix}}\circ(I\otimes\alpha(\tilde\omega))
\end{equation}
where, for an invertible matrix $M$, $C_{M}$ denotes the inner automorphism
of conjugation with $M$. Similarly, we see that
\begin{equation}
\mu(\kappa):=\Psi\circ\beta(\kappa)\circ\Psi^{-1}=
C_{\begin{pmatrix}
  0 & 1 \\
  1 & 0
\end{pmatrix}}\circ(I\otimes\alpha(\tilde\kappa))
\end{equation}
Observe the relation $(\mu(\kappa)\mu(\omega))^2=\operatorname{Id}$
(use
$(\alpha(\tilde\kappa)\alpha(\tilde\omega))^2=\alpha(\eta)$, the inner automorphism
of conjugation by $\eta$ on $\mathbf{H}_{e,c_V}$), showing that we indeed
defined a representation of $\W_{P,P}^D$, and not just of $G$.
Recall that $\bar{\sigma}(\delta)$ has the defining property that
$\mathbf{H}_e$ acts via $\Phi$ on
$\bar{\sigma}(\delta)(e_c)(V_{\bar{\sigma}(\delta)})\approx V_\delta\otimes V_\delta$
according to $\delta_+\otimes\delta_-$. It follows that
$\operatorname{Mat}_{2\times 2}(\bar{\mathbf{H}}_{e,c_V})=
\operatorname{Mat}_{2\times 2}(\mathbb{C})\otimes\bar{\mathbf{H}}_{e,c_V}$
acts via $\Psi$ on $V_{c,wc}:=\bar{\sigma}(\delta)(e_c+e_{wc})(V_{\bar{\sigma}(\delta)})\approx
\mathbb{C}^2\otimes (V_\delta\otimes V_\delta)$ by
$\operatorname{id}\otimes(\delta_+\otimes\delta_-)$ (here $\operatorname{id}$ denotes
the defining action of $\operatorname{Mat}_{2\times 2}(\mathbb{C})$ on $\mathbb{C}^2$).
We write elements of $\mathbb{C}^2\otimes (V_\delta\otimes V_\delta)$
as a column vector of size $2$ with entries in $V_\delta\otimes V_\delta$,
so that the action of $\operatorname{Mat}_{2\times 2}(\bar{\mathbf{H}}_{e,c_V})$
can be written as matrix multiplication where the matrix entries act on $V_\delta\otimes V_\delta$
via $\delta_+\otimes\delta_-$.

It follows straight from the definition of the correspondence that the factor set $\ga$
for $\W_{P,P}^D$ defined the by module $\bar\sigma(\delta)$ via the action $\beta$
is equal to the factor set defined by the module $V_{c,wc}$ via the action $\mu$.
The following linear involutions $M(\kappa), M(\omega)$ on $V_{c,wc}$
define intertwining operators from
$V_{c,wc}$ to its twists by $\mu(\kappa)$ and $\mu(\omega)$
respectively:
\begin{equation}
M(\kappa)\begin{pmatrix}
  u_1\otimes u_2 \\ v_1\otimes v_2
\end{pmatrix}
=
\begin{pmatrix}
  v_2\otimes v_1 \\ u_2\otimes u_1
\end{pmatrix};\
M(\omega)\begin{pmatrix}
  u_1\otimes u_2 \\ v_1\otimes v_2
\end{pmatrix}
=
\begin{pmatrix}
  -u_1\otimes \Omega(u_2) \\ \Omega(v_1)\otimes v_2
\end{pmatrix}
\end{equation}
We see that $(M(\kappa)M(\omega))^2=-\operatorname{Id}$, proving
that $M$ lifts to a linear representation of the Schur extension $G=\mathbb{D}_8$
of $\W_{P,P}^D$ in which $\eta$ acts by $-\operatorname{Id}$. In particular,
$\ga$ is nontrivial.
\end{proof}
\begin{cor} The 2-cocycle $\ga_\Delta$ of $\W_\Delta$ for $\H_n^D(\Lambda_{r,n})$ is nontrivial if $n>8$.
\end{cor}
\begin{proof} By definition (cf. paragraph \ref{subsub:2coc})
the pullback of $\ga_\Delta$ to
$\W_{P,P}\approx(\W_\Delta)_{(P,\bar\sigma(\delta)),(P,\bar\sigma(\delta))}$ is
equal to the pullback $(\pi^D)^*(\ga)$ of the factor set
$\ga$ via $\pi^D:\W_{P,P}\to\W_{P,P}^D$.
Using (\ref{eq:wppD}) it is easy to see that $\pi^D$ has a section
$s:\W_{P,P}^D\to\W_{P,P}$, implying that $(\pi^D)^*$ is injective
on cohomology classes by
contravariant functoriality of $H^2(?,\mathbb{C}^\times)$.
By the above Proposition
we conclude that $(\pi^D)^*(\ga)$ is a nontrivial $2$-cocycle, implying that
$\ga_\Delta$ is nontrivial.
\end{proof}
\section{Appendix: The Weyl groupoid}\label{app:weylgroupoid}
In this paragraph and in the next we recall some well known
facts about Weyl groups and standard parabolic subgroups of Weyl groups.
These results are
essentially due to Langlands, and the basic references for this
material are \cite{cas}, \cite{MW}.

Let $\mathfrak{a}=\operatorname{Lie}(T_\mathrm{rs})$ be the finite
dimensional real
vector space $\mathbb{R}\otimes_\Z Y$. Then $R_0\subset\mathfrak{a}^*$ is a
reduced, integral root system. Recall however that we do not assume
that $\mathbb{R}R_0=\mathfrak{a}^*$.

The Weyl group $W_0$ of $R_0$ acts naturally as a real reflection
group on $\mathfrak{a}$. The set of simple reflections in $W_0$ corresponding
to the basis of simple roots $F_0$ is denoted by $S_0$.

A parabolic subgroup of $W_0$ is the isotropy subgroup of an
element of $\mathfrak{a}$. A standard parabolic subgroup of $W_0$ is a
subgroup $W_P\subset W_0$ which is generated by the set of simple
reflections corresponding to a subset $P\subset F_0$. Clearly
every parabolic subgroup is conjugate to a standard parabolic
subgroup.

Let us denote by $\P$ the power set of $F_0$. Given $P,Q\in\P$ we
denote $\Wf_{P,Q}:=\{w\in W_0\mid w(P)=Q\}\subset W_0$.
\begin{dfn} The Weyl groupoid $\mathfrak{W}$ is the finite
groupoid whose set of objects is $\P$ and
$\operatorname{Hom}_\Wf(P,Q):=\Wf_{P,Q}$
\end{dfn}
For a standard parabolic subgroup $W_P$ we have a
distinguished set $W^P$ of left coset representatives for $W_P$,
characterized by $W^P:=\{w\in W_0\mid w(P)\subset R_{0,+}\}$.
We denote by $w_0$ the longest element of $W_0$, and by
$w_P$ the longest element of $W_P$. Then the longest
element in $W^P$ is equal to $w^P=w_0w_P$. Observe that
$\bar{P}:=w^P(P)\in\P$, so that we always have
$w^P\in\Wf_{P,\bar{P}}$. The element $\bar{P}\in\P$ is
called the conjugate of $P$.

If $P,Q\in\P$ and $\Wf_{P,Q}\not=\emptyset$ then $P,Q$ are called
associates. In particular, for every $P\in \P$ the conjugates
$P$ and $\bar{P}$ are associates.

Given $P\in\P$ we put $\mathfrak{a}^P=
\{x\in\mathfrak{a}\mid \a(x)=0\forall\a\in P\}$. Consider the
set $\mathfrak{S}:=\{(P,x)\mid P\in\P,\ x\in\mathfrak{a}^P\}$. Then
$\mathfrak{S}$ is a collection of real vector spaces which is
naturally fibred over $\P$. The set $\mathfrak{S}$ carries a
natural action of $\Wf$ defined by $w(P,x)=(Q,wx)$ if
$w\in\Wf_{P,Q}$.
\subsubsection{Chamber system of $\Wf$}\label{subsub:chsys}
We denote by $\mathfrak{a}^+$ the positive Weyl chamber in $\mathfrak{a}$.
Every face of $\overline{\mathfrak{a}^+}$ is of the from
$\overline{\mathfrak{a}^+}\cap\mathfrak{a}^P$
for a unique $P\in\P$, and this sets up natural bijection
between the facets of $\overline{\mathfrak{a}^+}$ and $\P$.

The subset of $R_0$ consisting of the roots of $W_P$ is denoted
by $R_P$, thus $R_P=R_0\cap \mathbb{R}P$. We choose the set of
positive roots $R_{P,+}$ in $R_P$ corresponding to the basis
$P$ of $R_P$.

We adopt the notation  $(P,\a)$ to denote the restriction of
$\a\in R_0\backslash R_P$ to $\mathfrak{a}^P$.
We write $R^P\subset\mathfrak{a}^{P,*}\backslash\{0\}$ for
the set of restrictions $(P,\a)$ of roots
$\a\in R_0\backslash R_P$ which are in addition primitive
in the sense that if $\b\in R_0\backslash R_P$ and
$(P,\a)\in R^P$ such that $(P,\a)$ and $(P,\b)$ are
proportional, then $(P,\b)=c(P,\a)$ with $c\in\mathbb{Z}$.
We write $R^P_+$ for the primitive restrictions corresponding
to the positive roots $\a\in R_{0,+}\backslash R_{P,+}$. An
element $(P,\a)$ is called {\it simple} if $(P,\a)$
is indecomposable in $\Z_+ R^P_+$. This is equivalent
to saying that $(P,\a)$ is the restriction of an element
of $F_0\backslash P$.

To each $(P,\a)\in R^P$ we associate a hyperplane
$(P,H_\a)=\operatorname{Ker}(P,\a)\subset\mathfrak{a}^P$. The hyperplanes
$(P,H_\a)$ are called the {\it walls} in $\mathfrak{a}^P$.

A chamber of $\Wf$ in $\mathfrak{S}$ is a pair
$(P,C)$ with $P\in\P$, and $C\subset\mathfrak{a}^P$ a connected component
of the complement of the collection of walls in $\mathfrak{a}^P$.
The collection of chambers is denoted by $C(\Wf,\mathfrak{a},F_0)$. This
is a finite set, which has a natural fibration $(P,C)\to P$
over the set $\P$. The action of $\Wf$ on $\mathfrak{S}$ maps
chambers to chambers, and thus induces a natural action
of the groupoid $\Wf$ on $C(\Wf,\mathfrak{a},F_0)$.

The set $R^P_+$ determines a distinguished chamber
$(P,\mathfrak{a}^{P,+})$ of $\mathfrak{a}^P$, defined by
$\mathfrak{a}^{P,+}=\{x\in\mathfrak{a}^P\mid \a(x)>0  \  \forall\a\in R^P_+\}$.
Observe that the chambers are simplicial cones, and that
$\overline{(P,\mathfrak{a}^{P,+})}$ is the face of
$\overline{\mathfrak{a}^+}$ corresponding to $P$.

An (irredundant) gallery of length $n$ in $\mathfrak{a}^P$ is a
sequence $C_0,C_1,\dots,C_n$ of chambers contained in $\mathfrak{a}^P$
such that each pair $C_{i-1},C_{i}$ ($i=1,\dots,n$) consists of distinct
chambers which share a common face. A minimal gallery is a gallery
of shortest length between its end points. The distance between
two chambers is the length of a minimal gallery between them.

Given a chamber $(P,C)$, we define its height
$\het(P,C)$ to be the number of walls of $\mathfrak{a}^P$
separating $(P,\mathfrak{a}^{P,+})$ and $C$. Thus $\het(P,C)$ is equal
to the distance between $(P,\mathfrak{a}^{P,+})$ and $C$.
\subsubsection{Elementary conjugations}\label{subsub:eltconj}
The faces of $\mathfrak{a}^{P,+}$ are
of the form $\mathfrak{a}^{Q,+}=\overline{\mathfrak{a}^{P,+}}\cap (P,H_\a)$,
where $Q\in\P$
is such that $P\subset Q$ and $Q\backslash P=\{\a\}$.
Thus the faces of $\mathfrak{a}^{P,+}$ are in bijective correspondence with
the
$Q\in\P$ containing $P$ as a maximal proper subset. Given
$Q\in\P$ containing $P$ as a maximal proper subset we define
an element $\s_Q^P\in\Wf_{P,P^\prime}$ by $\s_Q^P=w_Qw_P$.
Here $P^\prime=\s_Q^P(P)\subset Q$ is the conjugate of $P$
in $Q$. Notice that $\s_Q^{P^\prime}=(\s_Q^{P})^{-1}$.
In particular, $P\subset Q$ is self-opposed (in the
terminology of \cite[Section 10.4]{C}, i.e. $P\subset Q$ is
its own conjugate as a maximal standard parabolic subsystem of $Q$)
iff $\s_Q^P$ is an involution.
The following result is well known (see \cite{cas},\cite{MW}).
\begin{thm}\label{thm:cbs}
\begin{enumerate}
\item The action of $\Wf$ on the set
$C(\Wf,\mathfrak{a},F_0)$ of chambers of $\Wf$ is free, and
every $\Wf$-orbit in $C(\Wf,\mathfrak{a},F_0)$ contains a unique
positive chamber $(P,\mathfrak{a}^{P,+})$ (with $P\in\P$).
\item Every element $(P,x)\in\mathfrak{S}$ is
$\Wf$-conjugate to a $(Q,y)\in\mathfrak{S}$ with
$y\in\overline{\mathfrak{a}^{Q,+}}$.
\item If $(P,C_1)$ and $(P,C_2)$ are distinct neighboring chambers
then $C_1=w_1(\mathfrak{a}^{P_1,+})$, $C_2=w_2(\mathfrak{a}^{P_2,+})$
and $w_1^{-1}w_2$ is
the elementary conjugation $\s_Q^{P_2}$ in $\Wf$ with respect to
a uniquely determined $Q\in\P$ which contains both $P_2$ and
$P_1$ as maximal proper subsets.
\end{enumerate}
\end{thm}
\begin{cor}\label{cor:elcon}
\begin{enumerate}
\item Every $w\in\Wf_{P,Q}$ can be written as a product
of elementary conjugations in $\Wf$.
\item The minimal length of a word consisting of elementary
conjugations representing $w\in\Wf_{P,Q}$ is equal to the
height of $(Q,w(\mathfrak{a}^{P,+}))$.
\item The reduced expressions for $w\in\Wf_{P,Q}$
as a product of elementary conjugations correspond bijectively
to the minimal galleries in
$\mathfrak{a}^Q$ from $\mathfrak{a}^{Q,+}$ to $w(\mathfrak{a}^{P,+})$.
If $\mathfrak{a}^{Q,+}=C_1,C_2,\dots,C_n=w(\mathfrak{a}^{P,+})$ is a minimal gallery
with $C_i=w_i(\mathfrak{a}^{P_i,+})$ and we put
$x_i=w_{i-1}^{-1}w_{i}\in\Wf_{P_i,P_{i-1}}$,
then $w=x_1\cdot\dots \cdot x_n$ is the corresponding reduced
expression for $w$.
\end{enumerate}
\end{cor}

\end{document}